\documentclass[11 pt]{amsart}

\usepackage{hyperref}
\usepackage{amssymb, amsmath, amsthm, amsfonts}
\usepackage{verbatim}
\usepackage{bbm}
\usepackage{enumerate}
\usepackage{dsfont}
\usepackage{upgreek}
\usepackage[mathscr]{eucal}
\usepackage{tikz}
\usepackage[normalem]{ulem}

\usepackage{stackengine}
\newcommand{\ubar}[1]{\stackunder[1.2pt]{$#1$}{\rule{.9ex}{.075ex}}}

\usepackage[backgroundcolor=black!5,linecolor=black]{todonotes}
\usepackage{comment}

\usepackage[explicit,indentafter]{titlesec}
\titleformat{\section}{\centering\normalfont\scshape}{\thesection.}{.5em}{#1}
\titleformat{\subsection}[runin]{\normalfont\itshape}{\textnormal{\thesubsection.}}{.5em}{#1.}
\titleformat{\subsubsection}[runin]{\normalfont\itshape}{\thesubsubsection.}{.5em}{#1.}
\titlespacing{\section}{0em}{1em}{0.5em}
\titlespacing{\subsection}{0em}{.5em}{0.5em}

\newcommand{\ox}{{\bar x}}
\newcommand{\oy}{{\bar y}}
\newcommand{\ux}{{\ubar x}}
\newcommand{\uy}{{\ubar y}}

\usepackage{color}
\definecolor{gray}{gray}{0.5}

\newcommand{\cmt}[1]{}

\newcommand{\vertiii}[1]{{\left\vert\kern-0.25ex\left\vert\kern-0.25ex\left\vert #1 
    \right\vert\kern-0.25ex\right\vert\kern-0.25ex\right\vert}}
    
    \newcommand{\detail}[1]{}

\newcommand{\RH} {\rho_{\tiny{\mathrm{RH}}}}

	\renewcommand{\d}{(2*\n+\m)}
	
	\newcommand{\Qfourx}[1]{\d*(\d-1)/(\d*\d+\d*\m+(2*#1-1)*(\m+1))}
	\newcommand{\Qfoury}[1]{(\m+1)*(\d-1)/(\d*\d+\d*\m+(2*#1-1)*(\m+1))}
	\newcommand\Qthreex[1]{(\d-#1)/( \d-#1+\m+1)}
	\newcommand\Qthreey[1]{(\m+1)/(\d-#1+\m+1)}

	\newcommand{\definecoords}{
		\def\ptsize{.1pt}
		\def\QbgFillOpacity{.2}
		\def\QbgFillColor{black}
		\def\QbgLineOpacity{.6}
		\def\QbgDrawCritSeg{1}
		\def\QbgCritSegStyle{solid}
		\def\QbgCritSegOpacity{\QbgLineOpacity}
		\def\QbbFillOpacity{.1}
		\def\QbbLineOpacity{.5}
		\def\AFillOpacity{.1}
		\def\AFillColor{red}
		\def\ALineOpacity{.8}
		\coordinate (Q1) at (0,0);
		\coordinate (Q2) at ( {(2*\n-1)/(2*\n-1+\b)}, {(2*\n-1)/(2*\n-1+\b)}  );
		\coordinate (Q3) at ( {\Qthreex{\b}},  { \Qthreey{\b} }  );
		\coordinate (Q30) at ( {\Qthreex{0}}, { \Qthreey{0}  } );
		\coordinate (Q4b) at ( { \Qfourx{\b}  }, { \Qfoury{\b} }  );
		\coordinate (Q4g) at ( { \Qfourx{\g}  }, { \Qfoury{\g} }  );
		
		\coordinate (C1) at ( { (\Qfourx{\b}+\Qfourx{\g})/2 }, {(\Qfoury{\b}+\Qfoury{\g})/2}  ); 
		\coordinate (C2) at (Q4b);
	}

	\newcommand{\drawauxlines}{ 
		\draw (0,0) [->] -- (0,1) node [left] {$\frac1q$};
		\draw (0,0) [->] -- (1,0) node [below] {$\frac1p$};
		\draw [dashed,opacity=.3] (1,0) -- (0,1);
		\draw [dashed,opacity=.3] (0,0) -- (1,{(1+\m)/\d}); 
		\draw [dashed,opacity=.3] (0,0) -- (1,1); 
		\draw [dashed,opacity=.1]
		(.5, 0) -- (.5, .5);
	}

	\newcommand{\drawQbg}{
		\fill (Q1) node [left] {$Q_1$} circle [radius=.02em];
		\fill (Q2) node [above left] {$Q_{2}$} circle [radius=\ptsize];
		\fill (Q3) node [right] {$Q_{3}$} circle [radius=\ptsize];
		\fill (Q4g) node [below right] {$Q_{4}$} circle [radius=\ptsize];
		\fill [color=\QbgFillColor,opacity=\QbgFillOpacity] (Q1) -- (Q2) -- (Q3) -- (Q4g) -- cycle;
		\draw [opacity=\QbgLineOpacity] (Q1) -- (Q2) -- (Q3);
		\draw [opacity=\QbgLineOpacity] (Q4g) -- (Q1);
		\if\QbgDrawCritSeg1
			\draw [style=\QbgCritSegStyle,opacity=\QbgCritSegOpacity] (Q3) -- (Q4g);
		\fi
	}

	\newcommand{\ttf} {\tfrac{x'-y'}{t}}

\def\lc{\lesssim}
\def\gc{\gtrsim}

\def\eps{\varepsilon}
\def\bbone{{\mathbbm 1}}

\newcommand{\ci}[1]{_{{}_{\!\scriptstyle{#1}}}}

\newcommand{\Be}{\begin{equation}}
\newcommand{\Ee}{\end{equation}}

\newcommand{\Bm}{\begin{multline}}
\newcommand{\Em}{\end{multline}}

\def\intslash{\rlap{\kern  .32em $\mspace {.5mu}\backslash$ }\int}
\def\qsl{{\rlap{\kern  .32em $\mspace {.5mu}\backslash$ }\int_{Q_x}}}

\def\Re{\operatorname{Re\,}}
\def\Im{\operatorname{Im\,}}

\def\lc{\lesssim}
\def\gc{\gtrsim}

\def\emph#1{{\it #1 }}

\def\ga{\gamma}

\def\cf{{\it cf}}

\def\rank{{\mathrm{rank}}}

\def\supp{{\mathrm{supp}}}

\def\inn#1#2{\langle#1,#2\rangle}

\def\ga{\gamma}             
\def\eps{\varepsilon}

\def\ka{\kappa}
             \def\La{\Lambda}

\def\om{\omega}              
\def\vth{\vartheta}

\def\fI{{\mathfrak {I}}}
\def\fJ{{\mathfrak {J}}}

\def\fM{{\mathfrak {M}}}

\def\fS{{\mathfrak {S}}}

\def\fg{{\mathfrak {g}}}

\def\fk{{\mathfrak {k}}}

\def\fn{{\mathfrak {n}}}

\def\fs{{\mathfrak {s}}}

\def\fv{{\mathfrak {v}}}
\def\fw{{\mathfrak {w}}}
\def\fx{{\mathfrak {x}}}
\def\fy{{\mathfrak {y}}}
\def\fz{{\mathfrak {z}}}

\def\bbC{{\mathbb {C}}}

\def\bbH{{\mathbb {H}}}

\def\bbR{{\mathbb {R}}}

\def\bbV{{\mathbb {V}}}

\def\bbZ{{\mathbb {Z}}}

\def\cA{{\mathcal {A}}}

\def\cC{{\mathcal {C}}}

\def\cJ{{\mathcal {J}}}
\def\cK{{\mathcal {K}}}

\def\cM{{\mathcal {M}}}

\def\cR{{\mathcal {R}}}
\def\cS{{\mathcal {S}}}

\def\emph#1{{\it #1}}
\def\textbf#1{{\bf #1}}

\def\beq{\begin{equation}}
\def\endeq{\end{equation}}

\def\bs{\begin{split}}
\def\es{\end{split}}
\def\ybar{\overline y}

\theoremstyle{plain}

\newtheorem{thm}{Theorem}[section]
\newtheorem{prop}[thm]{Proposition}

\newtheorem{lem}[thm]{Lemma}
\newtheorem{cor}[thm]{Corollary}

\newtheorem*{thm*}{Theorem}
\newtheorem*{conj*}{Conjecture}
\newtheorem*{openproblem*}{Open Problem}

\theoremstyle{remark}
\newtheorem{rem}[thm]{Remark}

\newtheorem*{remarka}{Remark}
\newtheorem*{remarksa}{Remarks}

\numberwithin{equation}{section}

\definecolor{rscol}{rgb}{0,0,1.} 
\definecolor{jrcol}{rgb}{0.1,.6,0.1} 
\definecolor{ascol}{rgb}{1.,0,0} 

\def\R{\mathbb{R}}

\def\Z{\mathbb{Z}}

\def\H{\mathbb{H}}

\begin{document}
\title
[Spherical  maximal functions on Heisenberg groups]{ Lebesgue space estimates 
for spherical maximal functions on Heisenberg  groups}
\author[ J. Roos \  \ A. Seeger \ \  R. Srivastava]{Joris Roos \  \ \ \ Andreas Seeger \ \ \ \ Rajula Srivastava}

\address{Joris Roos: Department of Mathematical Sciences, University of Massachusetts Lowell, Lowell, MA 01854, USA, \& School of Mathematics, The University of Edinburgh, Edinburgh EH9 3FD, UK}
\email{jroos.math@gmail.com} 

\address{Andreas Seeger: Department of Mathematics, University of Wisconsin, 480 Lincoln Drive, Madison, WI, 53706, USA.}
\email{seeger@math.wisc.edu}

\address{Rajula Srivastava: Department of Mathematics, University of Wisconsin, 480 Lincoln Drive, Madison, WI, 53706, USA.}
\email{rsrivastava9@wisc.edu}

\thanks{Research supported in part by NSF grant DMS-1764295}
\date{\today}
\maketitle 

\begin{abstract} We prove  $L^p\to L^q$ estimates for local maximal operators associated with  dilates of codimension two spheres in Heisenberg groups; these are sharp up to two endpoints. The results  can be applied to improve currently known bounds on sparse domination for global maximal operators. We also consider lacunary variants, and extensions to M\'etivier groups. 
\end{abstract}

\section{Introduction}

Let $\bbH^n=\bbR^{2n} \times \bbR$ be the Heisenberg group of real Euclidean dimension $2n+1$. Writing $x=(\ux,x_{2n+1})$ with $\ux\in \bbR^{2n}$, the  group law is given by 
\[x \cdot y =(\ubar x+\ubar y, x_{2n+1}+y_{2n+1} + \ux^\intercal J\uy), \]
where $\ubar x^\intercal J\ubar y= \frac 12\sum_{j=1}^n(x_{n+j}y_j-x_jy_{n+j})$. A natural dilation structure on $\bbH^n$ is given by the  parabolic dilations
$\delta_t(x)=(t\ubar x, t^2 x_{2n+1}).$ 
These are automorphic, i.e. satisfy 
$\delta_t(x\cdot y)= \delta_tx\cdot\delta_t y$,  and  map the horizontal subspace $\bbR^{2n} \times\{0\}$ into itself.

Let $\mu$ be the normalized rotation-invariant  measure on the $2n-1$ dimensional sphere in $\bbR^{2n} \times \{0\}$, centered at the origin and let $\mu_t$ denote its $t$-dilate  defined by
$\inn{\mu_t}{f}= \inn{\mu}{f\circ\delta_t}$. 
The spherical means on the Heisenberg group,
\[f*\mu_t(x)=\int_{S^{2n-1}} f(\ubar x-t\om, x_{2n+1}-t 
\ux^\intercal J\om ) d\mu (\om)
\] were  introduced 
by Nevo and Thangavelu \cite{NevoThangavelu1997}. 
 In the theory of generalized Radon transforms, they  can be viewed as model   operators for which the incidence relation
(the support of the Schwartz kernel) has codimension two,
in contrast with the classical codimension one spherical means  (\cite{SteinPNAS1976,SchlagSogge1997,Cowling1980,GangulyThangavelu}).

The original interest in \cite{NevoThangavelu1997} was in pointwise convergence and ergodic results and hence  in $L^p$-estimates for the maximal function 
\[\fM f=\sup_{t>0} |f*\mu_t|.\]
A sharp result was proved by M\"uller and the second author \cite{MuellerSeeger2004} and, independently and by a different method,   by Narayanan and Thangavelu \cite{NarayananThangavelu2004}; namely for $n\ge 2$, the $L^p$ boundedness of $\fM$ holds if and only if $p>\frac{2n}{2n-1}$. It is conjectured that this statement holds true even when $n=1$ but this problem is currently still open (see \cite{BeltranGuoHickmanSeeger} for a recent positive result for $\fM$ acting on $L^p(
\bbH^1)$ functions of the form $x\mapsto f_\circ(|\ubar x|,x_3)$).

 Our renewed interest is prompted by recent work of Bagchi, Hait, Roncal and Thangavelu \cite{BagchiHaitRoncalThangavelu}, in which the authors consider $(p,q')$-sparse domination results for the operator $\fM$ with consequences for weighted inequalities \cite{bernicot-frey-petermichl}.  The primary ingredient in the proof of such a result is an induction argument relying on an  $L^p\to L^q$ estimate for the {\it local} maximal function
\[M f=\sup_{t\in I} |f*\mu_t|\] which is also of independent interest.  Here 
$I$ denotes a compact subinterval of $(0,\infty)$. 
The objective  is to find the best possible value of $q$ in such an inequality. For the sparse bounds one also needs to establish a closely related $\eps$-regularity  property, namely an $L^p\to L^q_\eps(L^\infty(I))$ estimate for the spherical means acting on compactly supported functions (\cf. \eqref{eq:eps-regularity result} below). 
If $q>p$ the operator norm in such estimates will depend on $I$ and it is no loss of generality to assume $I=[1,2]$. In \cite{BagchiHaitRoncalThangavelu} it is proved that $M$ maps $L^p(\bbH^n)$ to $L^q(\bbH^n)$ provided that $(\tfrac 1p, \tfrac 1q)$ belongs to the interior of the triangle with corners $(0,0)$, $(\tfrac{2n-1}{2n}, \tfrac{2n-1}{2n})$, $(\tfrac{3n+1}{3n+7}, \frac 6{3n+7})$. The authors ask  whether this result is essentially sharp, indeed results in 
\cite{Schlag1997}, \cite{SchlagSogge1997}, \cite{Lee2003} for the Euclidean analogues 
suggest that  it is not. In the following theorem we
provide $L^p\to L^q$ bounds that are sharp,  possibly except for  two endpoints at which we  prove restricted weak type inequalities. 
Implications on sparse bounds  for the global operator $\fM$ will be discussed in \S\ref{sec:sparse}, \cf. \eqref{sparse-bound}.

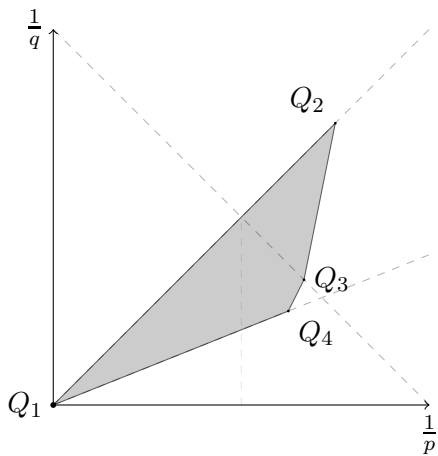
\begin{figure}[ht]
\begin{tikzpicture}[scale=5]
\def\m{1}
\def\n{2}
\def\b{1}
\def\g{1}
\definecoords
\drawauxlines
\drawQbg
\end{tikzpicture}
\caption{The region $\mathcal{R}$ in Theorem \ref{thm:maximal}, for $n=2$.}
\end{figure}
\begin{thm} \label{thm:maximal} Let $n\ge 2$.
Let  $\cR$ be the closed quadrilateral  with corners
\Be\label{quadrilateral}\begin{gathered}
Q_1=(0,0), \qquad Q_2=(\tfrac{2n-1}{2n}, 
\tfrac{2n-1}{2n}),  
\\
Q_3=(\tfrac{n}{n+1},\tfrac{1}{n+1}), \quad 
Q_4=(\tfrac{2n^2+n}{2n^2+3n+2},
\tfrac{ 2n}{2n^2+3n+2}).
\end{gathered} 
\Ee Then 

(i) $M$ is of restricted weak type $(p,q)$ for all $(\frac 1p, \frac 1q)\in \cR$.

(ii) $M:L^p(\bbH^n)\to L^q(\bbH^n)$ is bounded if $(\frac 1p, \frac 1q)$ belongs to the interior of $\cR$, or to the open boundary  segments $(Q_2,Q_3)$, $(Q_3, Q_4)$, or to the half open boundary segments $[Q_1,Q_2)$, $[Q_1,Q_4)$.

(iii) $M$ does not map $L^p(\bbH^n)$ to $L^q(\bbH^n)$  if $(\frac 1p, \frac 1q)\notin \cR$.

(iv) $M$ does not map $L^p(\bbH^n)$ to $L^p(\bbH^n)$   for  $(\tfrac 1p,\tfrac 1p)=Q_2$.
\end{thm}

We shall reduce the proof to estimates for standard oscillatory integrals of Carleson-Sj\"olin-H\"ormander type, in particular to a variant of  Stein's theorem  \cite{SteinBeijing} which was formulated in \cite{MSS93} and which relies on the maximal possible number of   nonvanishing curvatures for a cone in the fibers of the canonical relation.  It came as a surprise to the authors that such a simple reduction should be possible;  as far as we know this has not been observed for maximal functions associated with classes of generalized Radon transforms with  incidence relations of codimension two, or higher.
We shall now discuss cases with codimension greater than two.

\subsection*{Some extensions} 
We extend  Theorem \ref{thm:maximal} in two directions, already considered in \cite{MuellerSeeger2004}. One  extension deals with the situation  on $\bbH^n$ where the subspace $\bbR^{2n}\times \{0\}$ 
 is replaced by a general subspace transversal to the center; this tilted space is then no longer invariant under the automorphic dilations. Another extension is obtained by replacing the Heisenberg group with other two step nilpotent groups  
with higher dimensional center; here we will consider the class of M\'etivier groups  \cite{Metivier1980} which also includes the groups of Heisenberg type \cite{Kaplan1980}.

The Lie algebra $\fg$ of a two step nilpotent group $G$ splits as $\fg=\fw\oplus\fz$, so that  $[\fw,\fz]=\{0\}$ and $[\fw,\fw]\subset \fz$. The M\'etivier groups are characterized by a nondegeneracy condition, namely that
for every nontrivial linear functional $\upvartheta$ on $\fz$, the bilinear form
$\cJ^\upvartheta$ on $\fw\times \fw$ defined by $\cJ^\upvartheta=\upvartheta([X,Y])$ is nondegenerate.  This implies that $\fw$ is of even dimension. We set $\dim(\fw)=2n$, $\dim(\fz)=m$ and let $d=2n+m$ denote the Euclidean dimension of $G$. Identifying $\fw$ with $\bbR^{2n}$ and $\fz$ with $\bbR^m$, we  use exponential coordinates $x=(\ubar x, \bar x) \in \bbR^{2n}\times \bbR^m$;  the group multiplication is then given by 
\Be\label{eq:group-law} x\cdot y = (\ubar{x}+\ubar{y}, \bar{x}+\bar{y}+ \ubar{x}^\intercal  J \ubar{y}). \Ee
Here
 $\ux^\intercal   J\uy = \sum_{i=1}^m \ux^\intercal J_i \uy \,\bar  e_i \in \bbR^m$ with $\{\bar e_1,\dots,\bar e_m\}$ being the standard basis in $\bbR^m$ 
and $J_1,\dots,J_m$  denote skew symmetric matrices acting on $\bbR^{2n}$. The nondegeneracy condition on $\cJ^\upvartheta$ then says that for every $\theta\in\bbR^m\setminus\{0\}$, the $2n\times 2n$ matrix
$J^\theta = \sum_{i=1}^m \theta_i J_i $
is invertible. 
In the special case of groups of Heisenberg type we also have $(J^\theta)^2=-|\theta|^2 I$. We note that for every $m$ there are groups of Heisenberg type with  an $m$-dimensional center.   Kaplan \cite{Kaplan1980} points out the connection with Radon-Hurwitz numbers $\RH(k)$ defined as follows (\cite{hurwitz, radon}):  if  $k =(2\ell+1) 2^{4p+q}$ with $q\in \{0,1,2,3\}$ and for some $\ell\in  \{0,1,2,3,\dots\}$, then  $\RH(k) =8p+2^q$.  By \cite{radon}, \cite{Kaplan1980} there are 
($2n+m$)-dimensional groups 
of Heisenberg type with an $m$-dimensional center if and only if $m<\RH(2n)$. For odd $n$, we have $\RH(2n)=2$, hence  $m=1$.

The automorphic dilations on $G$ are  given by 
$\delta_t(\ux,\ox)=(t\ux, t^2\ox)$. 
We now let $\fv$ be a $2n$-dimensional subspace of  $\fg$ which is transversal to the center, i.e. in exponential coordinates \[V=\{(\ux,\Lambda \ux): \,\ux\in \bbR^{2n}\},\]
where $\Lambda$ is an $m\times 2n$ matrix with real entries.
Notice that $V$ is invariant under the dilation group $\{\delta_t\}$ only when $\La=0$.
Define a measure $\mu^\La\equiv\mu_1^\La $ supported on $V$ and its automorphic dilates $\mu_t^\La$  by
\[\inn{\mu_t^\La}{f}=\int_{S^{2n-1} } f(t\om, t^2\La \om) d\mu(\om).\] We consider the convolution  operator  $f\mapsto f*\mu_t^\La$  given explicitly by 
\[ f*\mu_t^\La(x)= \int_{S^{2n-1} } f(\ux-t\om, \ox-t^2 \La \om -t \ux^\intercal J\om) d\mu(\om)
\] and the associated local maximal function 
\Be \label{eq:MLa}M f = \sup_{t\in I} |f*\mu_t^\La|.\Ee These integral operators  
can be viewed as generalized Radon transforms associated to a family of surfaces of codimension $m+1$.
We note that when $m=1$ and $\La=0$, we recover the spherical means on the Heisenberg group considered in Theorem \ref{thm:maximal}. Let $\|\cdot\|$  denote the operator norm  of a matrix with respect to the  Euclidean norm. We state an extension of Theorem \ref{thm:maximal} under the assumption that $\La^\theta =\sum_{i=1}^m\theta_i\La_i$ is sufficiently small.

\begin{thm}\label{thm:main}
Let  $n\ge 2$, $d=2n+m$, let  $M$ be as in \eqref{eq:MLa}, and suppose that $\Lambda$ satisfies
\Be\label{eq:smallness-Lambda}\min_{\theta\in S^{m-1}} \big[ \|(J^\theta)^{-1}\|^{-1} - \|\Lambda^\theta\|\big] >0.
\Ee
Let $\cR$ be the closed quadrilateral with corners
\Be\label{viereck}
\begin{gathered} 
 Q_1=(0,0), \qquad Q_{2}=(\tfrac{d-m-1}{d-m}, \tfrac{d-m-1}{d-m}), \\
 Q_{3} = (\tfrac{d-1}{d+m}, \tfrac{m+1}{d+m}),\quad  Q_{4} = (\tfrac{d(d-1)}{d^2+(d+1)m+1}, \tfrac{(m+1)(d-1)}{d^2+(d+1)m+1}) .
 \end{gathered}
 \Ee
 Then 

(i) $M$ is of restricted weak type $(p,q)$ for all $(\frac 1p, \frac 1q)\in \cR$.

(ii) $M:L^p(\bbH^n)\to L^q(\bbH^n)$ is bounded if $(\frac 1p, \frac 1q)$ belongs to the interior of $\cR$, or to the open boundary  segments $(Q_{2},Q_{3})$, $(Q_{3}, Q_{4})$, or to the half open boundary segments $[Q_1,Q_{2})$, $[Q_1,Q_{4})$.

\end{thm}

\begin{remarksa} 
(i) For $m=1$ and $\La=0$, we recover the positive results in Theorem \ref{thm:maximal} for the spherical means on Heisenberg groups.

(ii) 
Under the smallness condition \eqref{eq:smallness-Lambda}, we give an alternative proof 
of the result for maximal operators  in \cite{AndersonCladekPramanikSeeger} which  relied on decoupling estimates to prove local  $L^p\to L^{p}_{(2n-1)/p'}$ regularity results for the averaging operators acting on compactly supported functions on $\bbH^n$, in the range $1<p<\frac{4n+2}{2n+3}$. Our approach is to use $L^p$ space-time estimates instead. However,  the $L^p$-Sobolev result in \cite{AndersonCladekPramanikSeeger} is interesting in its own right, and is still needed for the $L^p(\bbH^n)$ boundedness of the maximal operator in the range $p>\frac{2n}{2n-1}$ if one does not impose any condition on $\La$. The use of $L^2$ space-time estimates  is implicit already  in the work by Narayanan and Thangavelu 
\cite{NarayananThangavelu2004} who use  the group Fourier transform on the Heisenberg group to estimate a relevant square-function involving generalizations of spherical means. $L^2$ space-time estimates for the relevant Fourier integral operators have also been used in a more recent paper by Joonil Kim \cite{KimJoonil}; for the basic idea see also  the work on  variable coefficient Nikodym  estimates in \cite{MSS93}. 

(iii)
For $m+3\le 2n$, one  can use an alternative approach to the $L^{q'}\to L^q$ estimates which does not require  assumption \eqref{eq:smallness-Lambda}, \cf.  Remark \ref{rem:antidiag-max}. One also obtains the endpoint restricted weak type estimate for the point $Q_2$ provided that $m+3<2n$.

(iv) The example in \S\ref{sec:Q3Q4} demonstrating the sharpness of the line $Q_3Q_4$ in the case $m=1$ seems to be new. It would be interesting to see whether there exists similar examples for $m\ge 2$ and to settle the problem of sharpness for those cases. 
It  would also be interesting to analyze  what happens  when  the size restriction \eqref{eq:smallness-Lambda} on $\La$ is dropped.

\end{remarksa}

\subsection*{\texorpdfstring{$L^p$ improving estimates for  spherical averages}{Spherical averages}} We now discuss another problem considered in \cite{BagchiHaitRoncalThangavelu},  concerning sparse bounds  for the lacunary spherical maximal function $\sup_k|f*\mu_{2^k}|$ on the Heisenberg groups. Again, essentially sharp sparse bounds (\cf. \eqref{sparse-bound-lac}) follow from essentially sharp results on 
the $L^p(\bbH^n) \to L^q(\bbH^n)$ boundedness  for the averaging operators  $f\mapsto f*\mu$ 
and a closely related $\eps$-regularity property.
In \cite{BagchiHaitRoncalThangavelu} a partial result is proved;   namely the $L^p(\bbH^n)\to L^q(\bbH^n)$ boundedness   holds for $n\ge 2$ if $(\frac 1p,\frac 1q)$ belongs to the triangle with corners $(0,0)$, $(1,1)$ and  $(\tfrac{3n+1}{3n+4}, \frac {3}{3n+4})$; further, the method of \cite{BagchiHaitRoncalThangavelu} does not seem to yield a result for $n=1$.  
Here we prove  sharp results for all Heisenberg groups; indeed we formulate a general result for M\'etivier groups of dimension $d=2n+m$. 
\begin{thm}\label{thm:sphmeans}
(i) When $n\ge 2$ and $m<2n-2$,  the inequality 
\Be\label{sphmeans} \|f*\mu^\La\|_{L^q(G)} \lc \|f\|_{L^p(G) }\Ee  holds for all $f\in  L^p(G)$ if and only if $(\tfrac 1p, \tfrac 1q)$ belongs to the closed triangle $\triangle(P_1P_2P_3)$ with 
$P_1=(0,0)$, $P_2=(1,1)$ and $P_3=(\frac{2n+m}{2n+2m+1}, \frac{m+1}{2n+2m+1}).$

(ii) For $m=2n-2$, inequality \eqref{sphmeans}  holds if $(1/p,1/q) \in \triangle(P_1P_2P_3)\setminus P_3$. It fails for $(\tfrac1p,\tfrac1q)\notin \triangle(P_1P_2P_3)$.

(iii) For $m=2n-1$, inequality   \eqref{sphmeans}  holds if $(\tfrac1p,\tfrac1q)$ lies in the convex hull of 
$\{(0,0)$, $(1,1)$, $(\tfrac{4m^2+3m+1}{6m^2+5m+1}, \tfrac{m+1}{3m+1})$, 
$(\tfrac{6m+1}{9m+3}, \tfrac{3m+2}{9m+3})$, $(\tfrac {2m}{3m+1}, \tfrac{2m^2+2m}{6m^2+5m+1})$\}. This result is sharp at least when $m=1$. 
\end{thm}

In contrast to Theorem \ref{thm:main}, no assumption on $\La$ is needed; in fact, the estimate for convolution with  $\mu^\La_t$ (with {\it fixed $t$}) is equivalent to the corresponding inequality for $\La=0$, as one can see by a change of variable argument involving shear transformations. 

The above result be obtained using essentially known results on generalized Radon transforms and oscillatory integral operators with fold singularities (cf. \cite{PhongStein11991}, \cite{GreenleafSeeger1994}, \cite{Cuccagna1997}). 
For $m=1$ (the Heisenberg case)  the pentagon in (iii) reduces to a trapezoid and we get the sharp result
\begin{cor}\label{cor:sphmeans} The inequality 
\Be\label{Hsphmeans} \|f*\mu^\La\|_{L^q(\bbH^n)} \lc \|f\|_{L^p(\bbH^n) }\Ee  holds for all $f\in  L^p(\bbH^n)$ if and only if  one of the following holds:

(i) $n\ge 2$ and 
$(\tfrac 1p, \tfrac 1q)$ belongs to the closed triangle  with corners 
$(0,0)$, $(1,1)$ and $(\frac{2n+1}{2n+3}, \frac{2}{2n+3}).$

(ii)  $n=1$ and  $(\tfrac 1p, \tfrac 1q)$ belongs to the closed trapezoid with  corners 
$(0,0)$, $(1,1)$, $(\frac 23, \frac 12)$,  $(\frac 12,\frac 13)$.
\end{cor}

\begin{remarka} 
In view of the restriction $m<\RH(2n)$, only  cases with $m\le 2n-1$ occur in Theorem \ref{thm:sphmeans}. Observe that 
$\RH(4n+2)=2$,  and $\RH$ takes the values $4,8,4,8$ for $2n=4,8,12,16.$ 
Also 
$\RH(2n)<  2\log_2 (2n) +3$, hence clearly  $\RH(2n)\le 2n-2$ for $2n\ge 10$. 

The only cases with $m=2n-1$ are  
$(m,2n+m)=(1,3)$, $(7,15)$. In these instances,  the codimension $m+1$  of our sphere in $G$ exceeds half of the dimension of $G$. The first situation ($m=1$ and $2n=2$) corresponds to the Heisenberg group $\bbH^1$, for which the region in part (iii) of Theorem \ref{thm:sphmeans} is a trapezoid. In this case, we also establish  the sharpness of our result.

In the only two cases with   $m=2n-2$, namely $(m,2n+m)=(2,6)$,  $(6,14)$, we do not have a definitive answer for the  endpoint $P_3= (\tfrac{2n+m}{2n+2m+1}, \tfrac {m+1}{2n+2m+1})$.   
All endpoints in all the other cases are covered, since  part (i) of the theorem applies.
\detail{(ii) For the endpoint result at the corner $P_3$ one obtains an even stronger bound, namely that $\cA_t$ maps $L^{pr} (G) $ to $L^{p'r} (G)$ for $p=\frac{d+m+1}{d}$, and  $p\le r\le p'$.
For the  case $n=1$ we get stronger Lorentz space estimates at the corners 
 $(\frac 23, \frac 12)$ and   $(\frac 12,\frac 13)$, namely for $n=1$ the circular averaging operators $\cA_t$  map  
$L^{3/2,2}(\bbH^1) $ to $L^2(\bbH^1)$ and $L^2(\bbH^1)$ to $L^{3, 2} (\bbH^1)$.
On the edges improvements in the Lorentz scales follow simply by real interpolation.}

\detail{ (iii) The reader may  note that the numerology for the result for $n\ge 2$ is the same as in the nonvanishing rotational curvature case; this is due to the specific form of the fold surface in our example which leads to certain improvements in the $L^1\to L^\infty$ bound for  localizations near the fold surface (see the discussion in \S\ref{sec:proofoftheoremaverage}). Such an improvement  for localizations near the fold is also observed in the $L^p$-Sobolev  bounds  \cite{AndersonCladekPramanikSeeger} for the averaging operators (although the proof of those inequalities rely on more  complicated decoupling technology). }
\end{remarka}

\subsection*{Further directions}
It would also be interesting to investigate $L^p\to L^q$ mapping properties of maximal functions with respect to arbitrary dilation sets $E\subset[1,2]$ (see \cite{AHRS,RoosSeeger} for the Euclidean analogue of this question). We will  take up this problem in a subsequent paper \cite{rss2}.

\subsection*{Plan of the paper}  The proof of Theorem \ref{thm:main} is contained in the next three sections. In \S\ref{sec:main-results} we describe the basic estimates and how they can be reduced to problems about oscillatory integral operators. In \S\ref{sec:oscint} and \S\ref{sec:SToscest} we show how to apply in our context two well known theorems on oscillatory integral operators acting on $L^2$ functions.
Theorem \ref{thm:sphmeans} will be proved in \S\ref{sec:proofoftheoremaverage}. We establish the necessary conditions in  \S\ref{sec:sharpness} and \S\ref{sec:sharpness-II}. In \S\ref{sec:sparse} we briefly discuss the implications for sparse bounds.

\subsection*{Notation}
Partial derivatives will often be denoted by subscripts. 
$P$ denotes  the $(2n-1)\times2n$ matrix $P=(I_{2n-1}\,\, 0)$. \detail{Then $PJ$ is a $(2n-1)\times2n$ matrix with its rows identical to the first $2n-1$ rows of $J$.}
By $A\lesssim B$ we mean that $A\le C\cdot B,$ where $C$ is a constant and $A\approx B$ signifies that $A\lesssim B$ and $B\lesssim A$.
For coefficient vectors  $\bar y=(\bar y_1,\dots, \bar y_m)$, and sets of $1\times 2n$ vectors $\{\La_i\}_{i=1}^m$, or $2n\times 2n$ matrices  $\{J_i\}_{i=1}^m$,  we abbreviate $\La^{\bar y}=\sum_{i=1}^m \bar y_i\La_i$  and $J^{\bar y} =\sum_{i=1}^m \bar y_i J_i$. 

\section{Main estimates}
\label{sec:main-results} 
We use the notation $*_J$ for convolution  when the choice of $J$ in \eqref{eq:group-law} is emphasized. Let $\upsilon$ be a nonnegative bump function on $\bbR^{2n}$ supported in a neighborhood  of $e_{2n}$, normalized so that $\int_{SO(2n)} \upsilon(R^{-1}e_{2n}) dR=1$; here $dR$ denotes the normalized Haar measure on $SO(2n)$. Then we have $\int_{SO(2n)}\upsilon(R^{-1} \om) dR=1$ for all $\om\in S^{2n-1}$, and using this, Fubini's theorem and a change of variables,  we can write the convolution as 
\detail{\[ \int_{R\in SO(2n) } \int f_{R}(R^\intercal x-t\om, \bar x- t^2\La R\om -t  (R^\intercal x)^\intercal R^\intercal JR \om  ) \upsilon(\om)d\sigma(\om)
\] where $f_R(y)=f(Ry)$. }  
\[ f*\ci J \mu^\La_t(x)=\int_{R\in SO(2n)}  f_R*\ci{R^\intercal JR} [\upsilon \mu^{\Lambda R} ]_t  ( R^\intercal \ubar x,\bar x) dR,
\]
where $f_R(y)=f(R\ubar y,\bar y)$.
Note that replacing $(J_i,\La_i)$ with $(R^\intercal J_i R, \La_i R)$ does not affect condition \eqref{eq:smallness-Lambda} and therefore, by the integral Minkowski inequality, it suffices to prove our theorems with $\mu^\La$ replaced
 by $\upsilon\mu^\La$.
 
 By a localization argument we may assume that the function $f$ is supported in a small neighborhood of the origin. To see this we use the group translation to tile $G$. Let $Q_0=[-\tfrac 12,\tfrac 12)^{2n+m}$ and, for $\fn\in \bbZ^{2n+m}$, let $Q_{\fn}= \fn \cdot Q_0$, i.e.
 $Q_{\fn}=\{(\ubar \fn+\ubar z,\bar \fn+\bar z+\ubar \fn J\ubar z): z\in Q_0\}$.
 One then verifies that $\sum_{\fn\in \bbZ^{2n+m} }\bbone_{Q_\fn}=1$.
 Moreover, the measures $\mu_t$ are supported in $\{w\in G: |\ubar w|\le 2, |\bar w| \le 4\|\La\|\}$, hence in the union of $Q_\fk$ with 
 $|\fk_j|\le 2$, $j\le 2n$, $|\fk_{2n+i}|\le 2+4\|\La\|$, $i=1,\dots, m$. Denote this set of indices by $\fJ$. Then  \[\supp \big([f\bbone_{ Q_\fn}]*\mu_t\big) \subset 
 \bigcup_{\fk\in \fJ} (\fn\cdot Q_0\cdot Q_\fk) \subset \bigcup_{\tilde \fn\in \fI(\fn)} Q_{\tilde \fn},
 \]
 where $\fI(\fn)$ is a  set of indices $\tilde\fn$ with $|\fn_j-\tilde \fn_j|\le C(\La, J,n) $ for $j=1,\dots, 2n+m$. This consideration allows us to reduce to the case where $f$ is supported in a small neighborhood of the origin.
 
 Splitting $\ubar y=(y', y_{2n})$ and using the parametrization $\om=(w', g(w')) $ with $g(w')=\sqrt{1-|w'|^2}$ near the north pole $e_{2n}$ of the sphere,
 we are led to  consider the generalized Radon transforms associated to the incidence relation 
 given by the equations 
 \Be \label{eq:defining equations} y_{2n}= \fS^{2n}(x,t,y'), \quad \bar y= \bar \fS(x,t,\ubar y) \Ee where 
  \begin{subequations}
   \label{eq:Psidefb} 
   \begin{align}\label{eq:Psidefb1} 
\fS^{2n} (x,t,y')&= x_{2n}- tg (\tfrac{x'-y'}t)
\\
\label{eq:Psidefb2} 
\overline\fS (x,t,\underline y)&= 
\overline{x} +t \Lambda (\ubar{x}-\ubar{y}) +\underline{x}^\intercal J \underline{y},
\end{align}
\end{subequations}
where $y'$
is small, $\ubar x$ is near $e_{2n}$ on the support of $\upsilon$ and 
\begin{equation}\label{eq:Psidefc}
g(0)=1,\; \nabla g(0)=0,\; g''(0)= -I_{2n-1},\; g'''(0)=0.
\end{equation}
  Using \eqref{eq:defining equations} and \eqref{eq:Psidefb1} to express $y_{2n}$ in \eqref{eq:Psidefb2}, we conclude that \eqref{eq:defining equations} is equivalent with  \Be \label{eq:defining equations-2} \begin{aligned} y_{2n}&= \fs^{2n}(x,t,y'):=\fS^{2n}{(x,t,y')}, \\  \bar y&=  \bar \fs(x,t,y') 
:= \bar \fS(x,t,y', \fS^{2n}(x,t,y')).
 \end{aligned}
 \Ee
 Recall that   $P= \begin{pmatrix}I_{2n-1} &0\end{pmatrix}$. 
 We compute for $i=1,\dots, m$,
 \begin{align*} 
\overline \fs_i(x,t, y')=&\ox_i + t\La_i\ux 
-t\La_i P^\intercal y' 
+ \ux^\intercal J_i P^\intercal y'  
\\&
+\big(x_{2n}-tg(\tfrac{x'-y'}{t})\big) 
(\ux^\intercal J_i e_{2n}- t\La_ie_{2n})
\end{align*} and also have $\fs^{2n}(x,t,y') =x_{2n}-t g(\tfrac{x'-y'}{t}).$
  We can thus write, for $f$ with small support near $0$,
 \[ f*(\upsilon \mu^\La)_t (x) =
   \int \chi_1(x,t,y') f(y',\fs^{2n}(x,t,y'), \bar\fs(x,t, y') ) dy',\]
 where $\chi_1$ is a smooth and compactly supported function so that on its support $y'$ is small and $\ux$ is near $e_{2n}$.
The right hand side represents an operator  with  Schwartz kernel
\[K(x,t,y)= \chi_1 (x,t,y') \delta_0( \fs^{2n}(x,t,y') -y_{2n} , \overline\fs (x,t, y')-\overline y),\] where $\delta_0$ denotes the Dirac measure at  the origin in $\bbR^{m+1}$.
We express $\delta_0$ via the Fourier transform \Be\label{eq:delta=expansion}K(x,t,y) =
\chi_1 (x,t,y') \int_{\theta\in \bbR^{m+1}}
e^{i\psi(x,t,y,\theta)} \tfrac{d\theta}{(2\pi)^{m+1}} \Ee
with 
\Be\label{eq:Psidef} \psi(x,t,y,\theta) = \theta_{2n}(\fs^{2n}(x,t,y') -y_{2n} ) 
+\overline \theta\cdot (\overline\fs (x,t,y')-\overline y).
\Ee
Note that $K$ is well defined as an oscillatory integral distribution
(indeed from definition \eqref{eq:defining equations-2} we see that 
$x\mapsto K(x,t,y)$ and $y\mapsto K(x,t,y)$ are well defined as oscillatory integral distributions  on $\bbR^d$). 

We now perform a dyadic decomposition of this modified kernel. 
Let $\zeta_0$ be a smooth radial function on $\R^{m+1}$ with compact support in $\{|\theta|< 1\}$ such that $\zeta_0(\theta)=1$ for $|\theta|\le 1/2$. Setting $\zeta_1(\theta)=\zeta_0(\theta/2)-\zeta_0(\theta)$ and $\zeta_k(\theta)= \zeta_1(2^{1-k}\theta)$ for $k\ge 1$, we define 
\[ A^k_t f(x) = \int 
\chi_1(x,t,y')\int_{\theta\in \bbR^{m+1}}\zeta_k(\theta) 
e^{i\psi(x,t,y,\theta)} \tfrac{d\theta}{(2\pi)^{m+1}}\, f(y) dy  \]
and let  
\[ M^k f(x)=  \sup_{t\in [1,2]} |A^k_t f(x)|.\] The basic estimates for $M^k$ are summarized in the following proposition.
\begin{prop}\label{prop:Akmax} Assume \eqref{eq:smallness-Lambda} holds.

(i) For $1\le p\le \infty$, 
\Be\label{eqn:maxLpest}
\| M^k f\|_p \lc  2^{\frac{k}p} 2^{-k(d-m-1)\min(\frac1p,\frac1{p'})}\|f\|_p.\Ee
(ii) For $2\le q\le \infty$,
\Be \label{eqn:maxantidiagest}
\|M^k f\|_{L^q(\R^d)} \lesssim 2^{k (m+1-\frac{d+m}{q})} \|f\|_{q'}.\Ee
(iii) For $q\ge q_5:=\tfrac{2(d+1)}{d-1}$,
\Be\label{eqn:maxL2qest}
\| M^k f \|_{L^{q}(\R^d)} \lc 
2^{-k (\frac{d}{q} -\frac{m+1}{2}) }\|f\|_2.\Ee
\end{prop} 

\subsection{\it Proof of Theorem \ref{thm:main}, given Proposition \ref{prop:Akmax}} \label{sec:interpolpf}It suffices to show the required bounds for $\cM f(x):= \sum_{k\ge 0} M^k f$. 

We note that for $n\geq 2, m\geq 1$ (so $d\geq 5$) and $q_5:=\frac{2(d+1)}{d-1}$ we have $\tfrac{d}{q_5} -\tfrac{m+1}{2}>0$ and $m+1-\tfrac{d+m}{2}<0$.
To deduce the required restricted weak type estimates for $\mathcal{M}$ at $Q_2, Q_3, Q_4$
we recall the Bourgain  interpolation argument (\cite{Bourgain-CompteRendu1986}, \cite{CarberySeegerWaingerWright1999}):
Suppose we are given sublinear operators $T_k$ so that for $k\ge 1$,
\[ \|T_k\|_{L^{p_0,1}\to L^{q_0,\infty}} \lesssim 2^{k a_0}\;\text{and}\; \|T_k\|_{L^{p_1,1}\to L^{q_1,\infty}} \lesssim 2^{-k a_1}\]
for some $p_0, q_0, p_1, q_1\in [1,\infty], a_0, a_1>0$. Then the operator $\sum_{k\ge 1} T_k$ is of restricted weak type $(p,q)$, where 
\[ (\tfrac1p, \tfrac1q, 0) = (1-\vartheta) (\tfrac1{p_0},\tfrac1{q_0}, a_0) + \vartheta(\tfrac1{p_1},\tfrac1{q_1}, -a_1) \]
and $\vartheta=\tfrac{a_0}{a_0+a_1}\in (0,1)$. 

The restricted weak type estimate for $\mathcal{M}$ at $Q_2=(\tfrac{d-m-1}{d-m}, \tfrac{d-m-1}{d-m})$ now follows from
\eqref{eqn:maxLpest}. 
Similarly, the restricted weak type bound at $Q_3=(\tfrac{d-1}{d+m},\tfrac{m+1}{d+m})$ follows from \eqref{eqn:maxantidiagest}.
Finally, the restricted weak type bound at $Q_4=(\tfrac1{p_4},\tfrac1{q_4})$ with \[\tfrac1{p_4}=\tfrac{d(d-1)}{d(d-1)+(d+1)(m+1)},\,\tfrac1{q_4}=\tfrac{(m+1)(d-1)}{d(d-1)+(d+1)(m+1)}\]
follows from interpolating \eqref{eqn:maxL2qest} for $q=q_5$ with
the case $q=\infty$ of \eqref{eqn:maxantidiagest}, 
since for $n=d-m\ge 2$
\[(\tfrac{1}{p_4},\tfrac{1}{q_4},0)=\vartheta(\tfrac{1}{2},\tfrac1{q_5},-\tfrac{d}{q_5} +\tfrac{m+1}{2})+(1-\vartheta)(1,0,m+1)\]
with  $\vartheta=\tfrac{2(d+1)(m+1)}{d(d-1)+(d+1)(m+1)}\in (0,1)$. 
Since bounds for $\mathcal{M}$ imply bounds for $M$, 
this concludes the proof of part (i) of Theorem \ref{thm:main}. Part (ii) is immediate by interpolation.

\subsection{\it Reduction to space-time bounds}
We use  
the  standard   Sobolev inequality 
\Be \label {eq:Sobolev1D} \sup_{t\in[1,2] } |F(t)| \lc \|F\|_{p} + \|F\|_p^{1/p'} \|F'\|_{p}^{1/p}, \Ee where the $L^p$ norms are taken on $[1,2]$, see  \cite[p.499]{Stein-harmonic}. We apply it to $F(t)= A^k_tf(x)\equiv A^k f(x,t) $,  integrate in $x$ and then use H\"older's inequality to obtain\detail{the first summand on the right hand side is usually given by something like $|F(t_0)|$ but averaging the $pth$ powers in $t_0$ in $t_0$ yields the stated inequality.} Proposition \ref{prop:Akmax}  as  a consequence of the following
 \begin{prop} \label{prop:space-time-est} Assume \eqref{eq:smallness-Lambda} holds.
 
 (i) For $1\le p\le \infty$
 \Be \label{eqn:sptimeLpest}
\| {A}^k \|_{L^p(\bbR^d)\to L^p(\R^d\times [1,2])} \lesssim 2^{-k (d-m-1)\min(\frac1p,\frac1{p'})} 
\Ee
(ii) For $2\le q\le \infty$, 
\Be \label{eqn:sptimeantidiagest}
\|{A}^k f\|_{ L^{q'}(\bbR^d) \to L^q(\R^d\times [1,2])} \lesssim 2^{k (m+1-\frac{d+m+1}{q})} .
\Ee(iii) For $q\ge \frac{2(d+1)}{d-1}$,
\Be \label{eqn:Q5spacetime}
\| A^k \|_{L^2(\bbR^d)\to L^{q}(\R^d\times [1,2])} \lc 
2^{-k (\frac{d+1}{q} -\frac{m+1}{2}) }\|f\|_2. 
\Ee
(iv) The same estimates hold for  $2^{-k} \frac{d}{dt} A^k$ in place of $A^k$.
\end{prop}

For later calculations it will be convenient to introduce the  nonlinear shear transformation in the $x$-variables (smoothly depending on $t$)
\begin{align*} 
\underline \fx(x,t)&=\underline x, 
\\
\overline \fx_i (x,t)&=\ox_i 
-t\La_i\ux-x_{2n}({\underline x}^\intercal J_ie_{2n}-t\La_ie_{2n}).
\end{align*}
By a change of variables it suffices to prove the above space-time  inequalities for 
$A^k_t f(\fx(x,t),t) $ and 
$2^{-k} \frac{d}{dt} A^k_t f(\fx(x,t),t) $.
Using the homogeneity we see that both terms are linear combinations of expressions of the form  
\begin{equation}\label{eq:CalAk}  \mathcal{A}^k f(x,t) = 2^{k(m+1)} \int_{\bbR^d}\int_{\R^{m+1}} e^{i2^k \Psi(x,t,y,\theta)} b(x,t,y',\theta) f(y)\, d\theta\,dy 
\end{equation}
where 
the symbol $b$ is compactly supported in $\bbR^{d}\times\bbR\times \bbR^{d-m-1}\times \bbR^{m+1}$ with $y'$ near zero and $\ux$ near $e_{2n}$ on the support of $b$ and $|\theta|\in [1/2,2] $. The phase function $\Psi$ is given by \begin{equation}
\label{phasedefn}
\Psi(x,t,y,\theta)=\theta_{2n}(S^{2n}(x,t,y')-y_{2n}) +\sum_{i=1}^m  \bar\theta_i (\bar S_i (x,t,y')    -\bar y_i)
\end{equation}
with $(S^{2n}, \bar S)|_{(x,t,y')}= (\fs^{2n}, \bar\fs)|_{ (\fx(x,t),t,y')}$, that is 
\Be\label{eq:PhaseS}\begin{aligned}
S^{2n}(x,t,y')&= x_{2n}- tg(\ttf) 
\\
\bar S_i(x,t,y')&=x_{2n+i} +
(\ux^\intercal J_i -t\La_i) (P^\intercal y' -tg(\ttf) 
 e_{2n}),
\end{aligned}\Ee
with  $g$ is as in \eqref{eq:Psidefc}.
The Schwartz kernel of $\cA^k$ is given by 
\Be\label{eq:K_k}
\cK^k(x,t,y) = 
\int_{\bbR^{m+1}}e^{i2^k\Psi(x,t,y,\theta)} b(x,t,y',\theta) 
 d\theta \Ee and integration by parts yields the 
estimate
\Be\label{eq:Kkptw} 
|\cK^k(x,t,y)|\le C_N \frac{ 2^{k(m+1)} } {(1+2^k|y_{2n} - S^{2n} ( x, t,y')|+2^k |\bar y-\bar S (x,t,y')| )^N}.
\Ee
This estimate (together with the specific expressions for $S^{2n}$, $\bar S$) yields 
 for all $k\ge 1$   the bounds
\begin{gather}\label{eqn:trivestimate}
\|\cA_t^k\|_{L^1\to L^1} + \|\cA_t^k \|_{L^\infty\to L^\infty}  \lesssim 1,
\\ \label{eqn:trivL1infestimate}
\|\cA_t^k\|_{L^1\to L^\infty}  \lesssim 2^{k(m+1)}.
\end{gather}
In view of these estimates it suffices in what follows to consider the case of large $k$.
The bounds \eqref{eqn:sptimeLpest}, \eqref{eqn:sptimeantidiagest} then follow by an interpolation argument using 
\eqref{eqn:trivestimate}, \eqref{eqn:trivL1infestimate} and the local $L^2$ space-time estimate
\begin{equation}\label{eqn:sptimeL2estimate}
\| \cA^k f\|_{L^2(\R^d\times [1,2])} \lesssim 
2^{-k \frac{d-m-1}{2} }
\|f\|_2.
\end{equation}
This gives a gain over the  estimate 
$\| \cA^k_t\|_{L^2 \to L^2} \lesssim 
2^{-k (\frac{d-m-1}{2} -\frac 16) }$ for fixed time $t$  established in \cite{MuellerSeeger2004} via estimates for oscillatory integrals with fold singularities in  \cite{Cuccagna1997}. As mentioned before, the papers 
\cite{NarayananThangavelu2004} and \cite{KimJoonil} work with similar space-time estimates.

To prove \eqref {eqn:sptimeL2estimate} we use an oscillatory integral operator 
\[T_k f(x,t) = \int_{\R^d} e^{i 2^k \Phi(x,t,y)} b(x,t,y) f(y) dy\]
where $b\in C^\infty_c(\R^d\times\R\times\R^d)$ is as in \eqref{eq:CalAk}, and
\Be\label{eq:Phi-phase} \Phi(x,t,y)=
y_{2n}S^{2n}(x,t,y') +\sum_{i=1}^m  \bar y_i \bar S_i (x,t,y') .
\end{equation}
Setting $F_k(y)= \int_{\bbR^{m+1}} f(y',w_{2n}, \bar w) e^{-i2^k(y_{2n}w_{2n} +  \bar y\cdot \bar w)} dw_{2n} d\bar w$ we have 
\[ \cA^k f(x,t)= T_k F_k (x,t)\] 
and by Plancherel's theorem $\|F_k\|_2=(2^{-k}2\pi)^{(m+1)/2}\|f\|_2$.  Hence \eqref {eqn:sptimeL2estimate}
follows from 
\begin{prop} \label{prop:L2osc-spt} Assume \eqref{eq:smallness-Lambda} holds. For all $f\in L^2(\bbR^d)$, 
\Be\label{Tkest} 
\|T_k f\|_{ L^2(\bbR^d\times[1,2])} \lc 2^{-k\frac d2} \|f\|_2
\Ee
\end{prop} 
The proof will be given in \S\ref{sec:oscint}  using the standard H\"ormander  $L^2$ estimate (\cite{Hormander1973}). By the same argument, the $L^2\to L^q$ bound  \eqref{eqn:Q5spacetime} is reduced to the estimate
\begin{prop}
\label{prop:Q5spacetimeTk}
Assume that \eqref{eq:smallness-Lambda} holds. Then for  $q\ge q_5=\tfrac{2(d+1)}{d-1}$ and $f\in L^2(\bbR^d)$, 
\begin{equation}\label{eqn:Q5spacetimeTk}
\| T_k f\|_{L^{q}(\R^d\times [1,2])} \lc 
2^{-k \frac{d+1}{q}  }\|f\|_2.
\end{equation}
\end{prop}
This will be proved in \S \ref{sec:SToscest} using a result in \cite{MSS93}.

\detail{: {\color{blue} Not necessary anymore:}
It is advantagous to work with amplitudes in product form. Choose a  $C^\infty$ functions  $\chi_3$ on $\bbR^d\times \bbR$, $\chi_4$ on $\bbR^{2n-1}$ such that $\chi_3(x,t)\chi_4(y')=1$ on the support of $\chi_1$; we may assume that $\chi_4$ is supported near $0$ and that $|\ubar x|-t$ and $x'$ are small on the support of $\chi_3$.
By a Fourier series expansion of $\chi_1$ we may then express $\chi_1 $ as $\sum_{\ell_L\in \bbZ^{d+1} }\sum_{\ell_R\in \bbZ^{2n-1} } c(\ell_L, \ell_R)  e^{2\pi i \ell_L \cdot (x,t)} \chi_1(x,t)\chi_4(y') e^{2\pi i y\cdot \ell_R}$  with rapidly decaying coefficients.
By using Minkowski's inequality for the $\ell$-sums we see that for the sake of Lebesgue space estimate we can work with $\chi_3\otimes \chi_4$ in place of $\chi_1$.}

\section{Proof of Proposition \ref{prop:L2osc-spt}}
\label{sec:oscint} 
 By H\"ormander's classical $L^2$ bound
 (\cite[ch. IX.1]{Stein-harmonic}) applied after a partition of unity and a slicing argument  with a suitable subset of $d$ of the $(x,t)$-variables, it suffices to prove that the rank of the $(d+1)\times d$ mixed Hessian matrix $\Phi_{(x,t),y}''$ is equal to $d$.  Equivalently, for  
\Be \label{eq:Xidef}
 \Xi(x,t,y):=\nabla_{x,t} \Phi (x,t,y)=
 y_{2n} \nabla_{x,t}S^{2n} + \sum_{i=1}^m \oy_{i}  \nabla_{x,t}{\overline{S}}_{i}, 
 \Ee  we need to check that (using subscripts to denote partial derivatives) 
 \Be\label{eq:rank-condition} 
 \rank \, (\Xi_{y_1}, \dots,  \Xi_{y_d} )=d
 \Ee for
every $(x,t,y)\in \supp (b)$; in particular,  
  $|\oy|\approx 1$, and $x'-y'$ is small. 
  
  Recall that   $P$ denotes  the $(2n-1)\times2n$ matrix $P=(I_{2n-1}\,\, 0)$. 
  \detail{Then $PJ_i$ is a $(2n-1)\times2n$ matrix with its rows identical to the the first $2n-1$ rows of $J_i$. }
We calculate 
 \begin{multline*}\Xi(x,t,y)=
 y_{2n} \begin{pmatrix}
 -\nabla g(\ttf) \\1\\ \vec{0}_m\\ h(\ttf)
 \end{pmatrix}
 \\
 \,+\,\sum_{i=1}^m\overline{y}_i\begin{pmatrix}
  PJ_iP^\intercal y'  
  -tg(\ttf) PJ_ie_{2n}
  -(\ux^{\intercal} J_i e_{2n}-t\La_ie_{2n}) \nabla g(\ttf)
  \\ e_{2n}^\intercal J_i (P^\intercal y'-tg(\ttf) e_{2n})
  \\e_i^m
\\ 
h(\ttf)(\ux^{\intercal}J_i -t\La_i)e_{2n} -\La_i(P^\intercal y'-tg(\ttf) e_{2n})
 \end{pmatrix}
 \end{multline*}
 where $e_i^m$ denotes the $i$-th standard basis vector in $\mathbb{R}^m$ and \[ h(x')=\inn{x'}{\nabla g(x')}-g(x'),\]
 with 
 \Be\label{hderiv} h(0)=-1, \nabla h(0)=0, h''(0)= -I_{2n-1}.\Ee

The non-degeneracy assumption on $J$  implies that
$y_{2n}I_{2n}+ J_{\overline y}$ is invertible whenever 
$(y_{2n},\overline{y})\not=0$. More precisely, we have
the following auxiliary lemma for its operator norm (taken with respect to the standard Euclidean norm in $\bbR^{2n}$); this is a  quantitative extension of a lemma in \cite{MuellerSeeger2004}.

\begin{lem} \label{lem:skew-sym} Let $B$ be a real  skew-symmetric $N\times N $ matrix and let $I_{N}$ be the $N\times N$ identity matrix.

(i) Suppose $N$ is even. Then $\rho I_{N} +B$ is invertible if and only if either $\rho\neq 0$ or $B$ is invertible. Moreover, for the Euclidean operator norm of the inverse, 
\Be
\|(\rho I_{N} +B)^{-1}\|= \begin{cases} |\rho|^{-1} &\text{ if $\det B=0$} \\
(\rho^2+\|B^{-1}\|^{-2})^{-1/2} &\text{ if $\det B\neq 0$ }\end{cases}
\Ee

(ii)  Suppose that $N$ is odd. Then 
 $\rho I_{N} +B$ is invertible if and only if  $\rho\neq 0$ and we have
 $\|(\rho I_N+ B )^{-1} \|=|\rho |^{-1} $.  Moreover 
 $\det (\rho I_N+ B) = c(\rho,B) \rho  $ where $c$ depends smoothly on $\rho,B$  and  $c(\rho,B)\neq 0$
 if $\rank \, B=N-1$.
\end{lem}

\begin{proof} We first consider the case $N=2n$. 
When acting on $\bbC^{2n}$ the skew symmetric matrix $B$ has an orthonormal basis of eigenvectors, with purely imaginary eigenvalues. If $v$ is a complex eigenvector with eigenvalue $i\beta$, then $\bar v$ is an eigenvector with eigenvalue $-i\beta$, moreover 
 $B(\Re v)= -\beta  \Im v$      and $B(\Im v) = \beta \Re v.$ There is then  an orthonormal basis 
$u_1,\dots, u_{2n}$ of $\bbR^{2n}$   such that $Bu_{2k-1}= -\beta_k u_{2k}$   and $Bu_{2k}= \beta_k u_{2k-1}$.
 Also,  $(\rho I\pm B)u_{2k-1} =\rho u_{2k-1}\mp \beta_k u_{2k}$ and $(\rho I\pm B)u_{2k}=\beta_k u_{2k-1}\pm \rho u_{2k}$. Thus $\rho I\pm B$ are  invertible if and only if either $\rho\neq 0$ or $\min_k|\beta_k|\neq 0$.

We have  $((\rho I +B)^{-1} )^\intercal (\rho I+B)^{-1}= (\rho^2 I- B^2)^{-1} $,  by the skew-symmetry of $B$.
 Observe that  $\rho^2 I-B^2$ acts on 
 $\bbV_k:=\mathrm{span} \{u_{2k-1}, u_{2k}\}$ as $(\rho^2+\beta_k^2) I_{\bbV_k} $ and it follows that \begin{align*} \|(\rho I+B)^{-1}\|= \|((\rho I+B)^{-1})^\intercal (\rho I+B)^{-1} \|^{1/2} &
 \\=
 \max_{k=1,\dots, n}  (\rho^2+\beta_k^2)^{-1/2}= (\rho^2+ \min_{k=1,\dots,n} \beta_k^2)^{-1}.&\end{align*}
Since $\|B^{-1}\|^{-1}=  \min_{k} |\beta_k| $ we obtain the claimed expression for the operator norm.

Next consider the case $N=2n-1$, $n\ge 2$ (the case $N=1$ is trivial). The proof uses the same argument as above. We can now find an orthonormal bases $u_1,\dots, u_{2n-1}$ such that $Bu_{2k-1}= -\beta_k u_{2k}$   and $Bu_{2k}= \beta_k u_{2k-1}$ for $k=1,\dots, n-1$, and $Bu_{2n-1}=0$. 
Let $f(\rho, B) = \det (\rho I+B)$ then $f(0,B)=0$ since $N$ is odd, moreover $c(\rho,B)=f(\rho,B)/\rho$ is a polynomial in $\rho$ and the entries of $B$. 
By the above computation  $f(\rho,B)= \rho\prod_{k=1}^{n-1} (\rho^2+\beta_k)^2$. If  the rank of $B$ is $N-1$ then the $\beta_k$ are nonzero and thus $c(\rho,B)\neq 0$. \end{proof}

We proceed to check \eqref{eq:rank-condition}. Recall the notation   $J^{\overline{y}} = \sum_{i=1}^m \overline{y}_i J_i$ and $\La^{\overline y}=\sum_{i=1}^m \overline y_i \La_i.$
  We compute, for $j=1,\dots, 2n-1$, the partial derivatives (using  
 $e_{2n}^\intercal J_{\overline{y}} e_{2n}
=0$),
\[
 \Xi_{y_j}=\begin{pmatrix}
 t^{-1}(y_{2n}+(\ux^{\intercal}J^{\overline{y}}-t\La^{\overline y})e_{2n})\partial_j\nabla g(\ttf)+PJ^{\overline{y}}( e_j+\partial_jg(\ttf) e_{2n}) \\e_{2n}^\intercal J^{\overline{y}} (e_j +\partial_j g(\ttf)e_{2n})
 \\
 \vec{0}_m
 \\ -t^{-1}(y_{2n}+ (\ux^{\intercal}J^{\overline{y}}-t\La^{\overline y})e_{2n})\partial_jh(\ttf)
 -\La^{\oy}(e_j+\partial_jg(\ttf) e_{2n}) 
 \end{pmatrix},\]
  \[
\Xi_{y_{2n}}=\begin{pmatrix}
 -\nabla g(\ttf)\\1\\\vec{0}_m\\h(\ttf)
 \end{pmatrix},\] and, with $\ybar_i\equiv y_{2n+i}$,
 \[\Xi_{{y}_{2n+i}}=
 \begin{pmatrix}
  PJ_iP^\intercal y'  
  -tg(\ttf) PJ_ie_{2n}
  -(\ux^{\intercal} J_i e_{2n}-t\La_ie_{2n}) \nabla g(\ttf)
    \\ e_{2n}^\intercal J_i (P^\intercal y'-tg(\ttf) e_{2n})
  \\e_i^m
\\ 
h(\ttf)(\ux^{\intercal}J_i -t\La_i)e_{2n} -\La_i(P^\intercal y'-tg(\ttf) e_{2n})

 \end{pmatrix}.
\]
Let $\varPi : \bbR^{2n+m+1} \to \bbR^{2n+m}$ be the natural projection omitting the time variable $t=x_{2n+m+1}$.
Let  \Be\label{eq:sigma}\sigma\equiv\sigma(x,t,y)= y_{2n} +(\ux^\intercal J^{\oy}-t\La^{\oy})e_{2n}
\Ee 
and let $B=B(x,t,y)$ be the $(2n-1)\times(2n-1)$ matrix \[B=PJ^{\oy} e_{2n} (\nabla g(\ttf))^\intercal \]
with rank at most one.
We have 
\[
\varPi \Xi_y 
= \begin{pmatrix} 
t^{-1}\sigma g''(\ttf) + PJ^{\overline y} P^\intercal +B
&-\nabla g(\ttf)&*
\\
e_{2n}^\intercal J^{\overline y} P^\intercal&1&*
\\
0&0&I_m
\end{pmatrix} 
\]
and therefore (using elementary column operations and the skew symmetry of $J_{\oy}$) 
\begin{align} \label{eq:PiXiy} \det \varPi \Xi_y &=\det 
\begin{pmatrix}
t^{-1}\sigma g''(\ttf) + PJ^{\overline y} P^\intercal +B-B^\intercal
\end{pmatrix}\,.
\end{align}
Since  $PJ^{\overline y} P^\intercal +B-B^\intercal$ is a skew-symmetric $(2n-1)\times (2n-1)$ matrix, we see from Lemma \ref{lem:skew-sym} that  $\varPi \Xi_y$ is invertible if and only if  $\sigma\neq 0$. 
Equivalently $\det \Phi_{xy}''\neq 0$ if and only if $\sigma\neq 0$.  If $\sigma=0$, then 
$\varPi \Xi_y$ is not invertible and we have to use the $t$-derivatives.

\begin{rem} \label{kernel-cokernel-rem} For later reference in \S\ref{sec:proofoftheoremaverage} we include the following remarks which  establish the oscillatory integral operator $f\mapsto T_k f(\cdot,t)$ as an operator with a folding canonical relation (i.e.  two-sided fold singularities).  
We examine the  one-dimensional kernel and cokernel of the matrix in \eqref{eq:PiXiy}  for $x'=y'$, $\sigma=0$. 

(i) Consider  $b=(b',b_{2n}, \bar b)^\intercal $ in the kernel.  Then $\bar b=0$, $b_{2n} = -e_{2n}^\intercal J^{\bar y}  P^\intercal b'$ and  $PJ^{\bar y} P^\intercal b'=0$ with $b'\neq 0$. This also implies that $e_{2n}^\intercal J^{\bar y} P^\intercal b'\neq 0$ (since otherwise $P^\intercal b'$ would be in the kernel of the invertible matrix $J^{\bar y}$ and $b'$ would be zero).  
Let $V_L = \sum_{j=1}^{2n-1} b_j \partial/\partial{y_j} + b_{2n} \partial/\partial{y_{2n} }$ with $b_{2n} = - e_{2n}^\intercal J^{\bar y}  P^\intercal b'\neq 0$, then
 $\tfrac{\partial\sigma}{\partial y_{2n}} =1$ and from part (ii) of Lemma \ref{lem:skew-sym} we get 
 $V_L(\det \varPi\Xi_y) \neq 0$.
 
 (ii) Let $a^\intercal=(a_1,\dots, a_{2n+m})$ be in the cokernel of $\varPi\Xi_y$. 
 Then a right kernel vector field $V_R = \sum_{j=1}^{2n+m} a_j \partial/\partial{x_j} $ satisfies $a_{2n}=0$ (when evaluated at $x'=y'$, $\sigma=0$), $PJ^{\bar y} P^\intercal a'=0$ with $a'\neq 0$ and $\bar a$ is determined by $a'$. Note that
 $V_R \sigma =- e_{2n}^\intercal J^{\bar y} P^\intercal a'\neq 0$ which leads to 
 $V_R (\det \varPi\Xi_y) \neq 0 $.
 
\end{rem}
\subsection*{The case of small $\sigma$} 
 We consider $\bar y$ in an $\eps$-neighborhood of $\bar y_\circ \neq 0$, with small $\eps>0$. 
 We look at the $(2n+m+1)\times (2n+m)$ matrix $\Xi_{y}$ for $x'=y'$ and small $\sigma$. Since $g''(0)=-I_{2n-1}$ and $h(0)=-1$, we have
  \[
 \Xi_y \Big|_{x'=y'} 
= \begin{pmatrix} 
 -t^{-1}\sigma   I_{2n-1} + J^{\overline y} P^\intercal
&0&*
\\
e_{2n}^\intercal J^{\overline y} P^\intercal&1&*
\\
0&0&I_m
\\
-\La^{\bar y}P^\intercal &-1 &*
\end{pmatrix} 
\]
 We recall that $J^{\bar y} $ is invertible for $\bar y\neq 0$ (and take $\bar y$ near $\bar y_\circ\neq 0$).
This implies that for $\widetilde \La^{\bar y}= (\La^{\bar y} P^\intercal, 0)$
the $2n\times 2n$ matrix with rows $e_1^\intercal J^{\bar y}$, \dots,  $e_{2n-1}^\intercal J^{\bar y}$, $e_{2n}^\intercal J^{\bar y} -\widetilde \La^{\bar y} $
is invertible.  To see this, let $\{c_k\}_{k=1}^{2n}$ be such that 
$\sum_{k=1}^{2n-1} c_k e_k^\intercal J^{\bar y} + c_{2n} (e_{2n}^\intercal J^{\bar y} -\widetilde \La^{\bar y} )=0$. 
This gives 
 $\sum_{k=1}^{2n} c_k e_k^\intercal = c_{2n} \widetilde \La^{\bar y} (J^{\bar y})^{-1} $
and thus $\sum_{k=1}^{2n-1} c_k^2+ c_{2n}^2(1- \|\widetilde \La^{\bar y} (J^{\bar y})^{-1}\|^2)=0 $. By \eqref{eq:smallness-Lambda}, setting    $\theta=\bar y/|\bar y|$,
\[\|\widetilde \La^{\bar y} (J^{\bar y})^{-1}\| \le \|\widetilde \La^{\bar \theta} \| \|(J^{\bar \theta})^{-1} \| \le  \| \La^{\bar \theta} \| \|(J^{\bar \theta})^{-1} \| < 1\]  which implies that the $c_k$ are all zero.

The preceding consideration also yields that $2n-1$ of the truncated rows
$e_1^\intercal J^{\bar y}P^\intercal $, \dots, $e_{2n-1}^\intercal J^{\bar y}P^\intercal$, $e_{2n}^\intercal J^{\bar y}P^\intercal  -\La^{\bar y} P^\intercal $ are linearly independent. 
For $\kappa \in\{1,\dots, 2n-1\}$,  we let
$P^{(\ka)} :\bbR^{2n-1} \to \bbR^{2n-2}$  denote the map  that omits the   
   $\kappa^{\mathrm{th}}$ coordinate.  We also let   $\varPi^{(\ka)}$ be the corresponding linear map from $\bbR^{2n+m+1}$ to  $\bbR^{2n+m} $  that omits the $\ka^{\mathrm{th}}$ coordinate.
  
Since the skew symmetric matrix $PJ^{\bar y}P^\intercal$ is not invertible we see that there is a $\kappa \in \{1,\dots, 2n-1\}$ (depending on $\bar y$)  such that the
$(2n-1)\times (2n-1)$ matrix 
\[
 \begin{pmatrix} P^{(\kappa)}P J^{\bar y} P^\intercal  
\\
(e_{2n}^\intercal J^{\bar y} -\La^{\bar y})  P^\intercal  
\end{pmatrix} \]
is invertible. By elementary row operations this implies that 
\[
 \varPi^{(\ka)}  \Xi_y \Big|_{x'=y'} 
= \begin{pmatrix} 
 -t^{-1}\sigma  P^{(\ka)} PI_{2n} +P^{(\ka)} PJ^{\overline y} P^\intercal
&0&*
\\
e_{2n}^\intercal J^{\overline y} P^\intercal&1&*
\\
0&0&I_m
\\
-\La^{\bar y}P^\intercal &-1 &*
\end{pmatrix} 
\]
is invertible for $\sigma=0$.  The above calculations  for $\bar y=\bar y_\circ$, $\sigma=0$ and $x'=y'$ extend by continuity to small choices of 
$|\sigma|$,  $|x'-y'|$ and $|\bar y-\bar y_\circ|$, and for these we obtain that  $\varPi^{(\ka)}  \Xi_y$ is invertible.
This concludes the verification of \eqref{eq:rank-condition} and thus the proof of Proposition \ref{prop:L2osc-spt}. \qed

\section{Proof of Proposition \ref{prop:Q5spacetimeTk}} \label{sec:SToscest}
Let $\Xi=\nabla_{x,t}\Phi$ as in \eqref{eq:Xidef},  $N\in\R^{d+1}$ be a unit  vector, and let $\mathscr C^N(x,t,y)$ be the $d\times d$ curvature matrix with respect to $N$ given by
\Be\label{curvmatrix}  \mathscr C_{jl}^N = \frac{\partial^2}{\partial y_j\partial y_l} \inn{ N}{\Xi}  \Ee
We apply an oscillatory integral result  in \cite{MSS93} according to which Proposition \ref{prop:Q5spacetimeTk} holds provided that \eqref{eq:rank-condition}  and  the additional  curvature condition
\Be \label{eq:rank-curv}\inn {N}{ \Xi_{y_j}  }=0 ,\,\, j=1,\dots, d \quad \implies \quad \rank\, \mathscr C^N = d-1\Ee 
is satisfied; i.e. the conic surface  $\Sigma_{x,t}$ parametrized by $y\mapsto \Xi(x,t,y)$  has the maximal number $d-1$ of nonvanishing principal curvatures. 
It remains to verify \eqref{eq:rank-curv}; here we shall use our size assumption \eqref{eq:smallness-Lambda} on $\La$.

Let  $\sigma$ be as in \eqref{eq:sigma}. 
For $x'=y'$, 
using the properties of $g,h$ in 
\eqref{eq:Psidefc}, \eqref{hderiv} we get 
\begin{align*}
 \Xi_{y_j}
 \Big|_{x'=y'}
 &=-t^{-1}\sigma e_j+\begin{pmatrix}
 J^{\overline{y}}e_j\\\vec{0}_m\\ -\La^{\oy}e_j
 \end{pmatrix},\\
\Xi_{y_{2n}} \Big|_{x'=y'}&=\begin{pmatrix}
 \vec{0}_{2n-1}\\1\\\vec{0}_m\\-1
 \end{pmatrix},\,\, \Xi_{{y}_{2n+i}} \Big|_{x'=y'}=\begin{pmatrix}
  PJ_iP^\intercal y' -t PJ_ie_{2n}\\
  e_{2n}^\intercal J_i P^\intercal y'\\e_i^m\\ (t\La_i-x^\intercal J_i)e_{2n}-\La_i(P^\intercal y'-te_{2n})
 \end{pmatrix}.
\end{align*}

We now consider a unit vector \[N\Big|_{x'=y'}= (\underline{\alpha},\overline{\alpha},\alpha_{d+1})^\intercal= (\alpha',\alpha_{2n},\overline{\alpha},\alpha_{d+1})^\intercal\in \mathbb{R}^{d+1}\] perpendicular to 
 $\Xi_{y_i}$, $\Xi_{y_{2n}} $, $\Xi_{\overline{y}_i}$. 
Evaluating  for  $x'=y'$, we get
 \begin{subequations}
\begin{align}  \label{normal1}
 0&=\inn {N}{\Xi_{y_j} }\Big|_{x'=y'}=-t^{-1}\sigma
 \alpha_j+\underline{\alpha}^{\intercal}J^{\bar{y}}e_j-\alpha_{d+1} \La^{\oy} e_j, \quad j\le 2n-1.
 \\
 \label{normal2}  0&=\inn{N} {\Xi_{y_{2n} }}\Big|_{x'=y'}=\alpha_{2n}-\alpha_{d+1},
\\
\label{normal3}
0&=\inn {N} {\Xi_{\overline{y}_i} }\Big|_{x'=y'}=\, 
\alpha'^{\intercal}(PJ_iP^\intercal y' -t  PJ_i e_{2n})
+\alpha_{2n} e_{2n}^\intercal J_i P^\intercal y' +\overline{\alpha}_i \\ \notag
&\qquad\qquad \qquad +\alpha_{d+1}((t\La_i-x^\intercal J_i)e_{2n} -\La_i(P^\intercal y'-te_{2n})), \,\, i=1,\dots,m.
\end{align}
\end{subequations}
Equation \eqref{normal3} above expresses $\overline{\alpha}_i$ in terms of $\underline{\alpha}$ and $\alpha_{d+1}$ and turns out to be  not really  relevant to our calculations. Normalizing $|N|=1$ we have $|\underline\alpha|\approx 1$. 

\begin{remarka} It is instructive to see that when  $\Lambda=0$ and for the special case of the Heisenberg type group, i.e. when $(J^{\oy})^2=-|\oy|^2 I$,    the projection of the normal vector $N$ to $\bbR^{2n}$  is  tangential to the sphere  for $\sigma=0$, indeed in that case (as we evaluate at the northpole of the sphere with normal vector $e_{2n}$) we see  from  \eqref{normal1} 
 that $\ubar\alpha$ is perpendicular to $\mathrm{span}\{J^{\bar{y}} e_1,\dots J^{\bar y} e_{2n-1}\}$ which contains $e_{2n}$. 
\end{remarka}

The  second derivative vectors are given by
\[
 \Xi_{y_jy_k}=\begin{pmatrix}
 -t^{-2}\sigma\partial_{jk}\nabla g(\ttf)-t^{-1}(PJ^{\overline{y}}e_{2n})\partial_{jk}^2g(\ttf) \\0\\\vec{0}_m\\ t^{-2}\sigma\partial_{jk}^2h(\ttf)+t^{-1}\La^{\oy}e_{2n}\partial_{jk} g(\ttf) 
 \end{pmatrix},\] for $1\le j,k\le  2n-1$,  
 and
\[ \Xi_{y_{j}y_{k}} = 0, \text{ if } 2n\le j, k\le 2n+m.
\]
Moreover,
 \[
\Xi_{y_jy_{2n}}=\begin{pmatrix}
 t^{-1}\partial_j\nabla g(\ttf)\\0\\\vec{0}_m\\-t^{-1}\partial_jh(\ttf)
 \end{pmatrix},\quad 1\le j\le 2n-1,\] and, for $i=1,\dots, m$ and  $j=1,\dots, 2n-1$, 
 \[\Xi_{y_j{y}_{2n+i}}=\begin{pmatrix}
  PJ_ie_j+ PJ_ie_{2n}\partial_jg(\ttf)+t^{-1}(\ux^{\intercal} J_i -t\La_i)e_{2n}\partial_j\nabla g(\ttf) \\ e_{2n}^{\intercal}J_i e_j\\\vec{0}\\ 
  - t^{-1}(\ux^{\intercal}J_i-t\La_i)  e_{2n}\partial_j h(\ttf) -\La_i(e_j+\partial_j g(\ttf) e_{2n}) 
 \end{pmatrix}, 
\] 
\detail{
For $x'=y'$ we get (using  
$g''(0)=h''(0)=-I_{2n-1}$, $g'''(0)=0$)
 for  $1\le j,k\le 2n-1$, 
\begin{align*}
 \Xi_{y_jy_j} \Big|_{x'=y'}&=\begin{pmatrix}
 t^{-1}PJ_{\overline{y}}e_{2n}\ \\0\\\vec{0}_m\\ -t^{-2}
 \sigma- t^{-1}\La_{\oy} e_{2n}
 \end{pmatrix},\\ \Xi_{y_j y_k}\Big|_{x'=y'} &=0 \text{ if } j\neq k, \qquad 
\Xi_{y_jy_{2n}} \Big|_{x'=y'}=
 -t^{-1}e_j,\\ 
 \Xi_{y_j{y}_{2n+i}} \Big|_{x'=y'}&=\begin{pmatrix}
  J_ie_j \\\vec{0}_{m}\\-\La_ie_j
 \end{pmatrix}-t^{-1}((\ux^{\intercal} J_i -t\La_i)e_{2n})e_j 
\end{align*}
and
\[ \Xi_{y_{2n}y_{2n}} = \Xi_{\overline{y}_i\overline{y}_i'}=\Xi_{y_{2n}\overline{y}_i}  =\vec 0 \text{ when $x'=y'$.}\] 
}
We evaluate at $x'=y'$,  using $g''(0)=h''(0)=-I_{2n-1}$, $g'''(0)=0$, and  see that the  components 
of the curvature matrix $\mathscr C^N$ at $x'=y'$ 
are given by
\begin{align*}
\inn{N}{\Xi_{y_jy_j}}\Big|_{x'=y'}&=(\alpha')^\intercal t^{-1}PJ^{\bar{y}}e_{2n}  -\alpha_{d+1}t^{-2}\sigma
+\alpha_{d+1} (-t\La^{\oy}e_{2n}),\\
&=(\ubar \alpha)^\intercal t^{-1}J^{\bar{y}}e_{2n}  -\alpha_{d+1}(t^{-2}\sigma
+ t^{-1}\La^{\oy}e_{2n}),
\\
\inn{N}{\Xi_{y_jy_k}}\Big|_{x'=y'}&=0,\quad \text {if } j\neq k,  
\end{align*} 
for $1\le j,k\le 2n-1$. Moreover  for $1\le j\le 2n-1$,  
\begin{align*}
\inn{N}{\Xi_{y_jy_{2n}}}\Big|_{x'=y'}&=-\alpha_j t^{-1},\\
\inn{N}{\Xi_{y_j{y}_{2n+i}}}\Big|_{x'=y'}&=\ubar{\alpha}^{\intercal} J_i e_j- \alpha_j t^{-1}((\ux^{\intercal}J_i-t\La_i)e_{2n})-\alpha_{d+1}\La_i e_j, \, 1\le i\le m,
\end{align*}
and
\[ \inn{N}{\Xi_{y_{j}y_{k}}}\Big|_{x'=y'}= 0, \quad 2n\le j,k\le d
=0.\]

Thus we get for the $d\times d $ curvature matrix $\mathscr C^N$,
\[\mathscr C^N\Big|_{x'=y'}=\begin{pmatrix}
cI_{2n-1} &PA
\\
A^{\intercal}P^\intercal& 0
\end{pmatrix}
\]
where $c=c(t,x,y)$ is given by
\begin{equation}
\label{cdefn}
c=t^{-1}\ubar \alpha^\intercal J^{\oy}e_{2n} -t^{-2}\alpha_{d+1} \sigma
-t^{-1}\alpha_{d+1}\La^{\oy}e_{2n}
\end{equation}
and $A^\intercal P^\intercal$ is the $(m+1)\times (2n-1)$ matrix obtained from the following $(m+1)\times 2n $ matrix $A^\intercal $ by deleting the last column; 
here
\[A^\intercal=\begin{pmatrix}
-t^{-1}(\ubar \alpha)^\intercal
\\
\ubar {\alpha}^{\intercal} J_1 -t^{-1}((\ux^{\intercal}J_1-t\La_1)e_{2n})\ubar  \alpha^{\intercal}-\alpha_{d+1}\Lambda_1  
\\
\vdots
\\
\ubar {\alpha}^{\intercal} J_m -t^{-1}((\ux^{\intercal}J_m-t\La_m) e_{2n})\ubar \alpha^{\intercal}-\alpha_{d+1} \Lambda_m 
\end{pmatrix}.\]
We combine \eqref{cdefn}, \eqref{normal1} and \eqref{normal2} to get
\Be\label{c-nonzero}
(-t^{-1}\sigma I+J^{\oy} )\ubar\alpha  -\alpha_{2n}(\La^{\oy})^\intercal 
= ct e_{2n}.\Ee
Therefore, by Lemma \ref{lem:skew-sym}, and writing $\vth=\oy/|\oy|$, 
\begin{align*}
 |c| &= t^{-1} \|(\tfrac{\sigma}{ t} I-J^{\bar y} )\ubar \alpha +\alpha_{2n}\La^{\oy}\|
 =
  t^{-1} |\bar y| \|(\tfrac{\sigma }{|\oy| t} I-J^{\bar \vth })\ubar \alpha +\alpha_{2n}\La^{\bar\vth}\|
 \\ \notag 
&\ge t^{-1}|\bar y| |\ubar \alpha|\big(
 (\tfrac{\sigma^2 }{|\oy|^2 t^2}+\|(J^{\bar \vth })^{-1}\|^{-2})^{1/2} - \|\La^{\bar\vth}\|
 \big)
 \end{align*}
 and thus \Be \label{c-lowerbound} |c| \ge t^{-1}|\bar y||\ubar \alpha|(\|(J^{\bar \vth })^{-1}\|^{-1} - \|\La^{\bar\vth}\|) \Ee which is bounded away from zero by assumption \eqref{eq:smallness-Lambda}.
 
We finish by verifying that $\mathscr C^N$ has rank $d-1=2n+m-1$.
We have the factorization
\begin{multline}\label{eq:factorization}
\begin{pmatrix} c I_{2n-1} &PA\\ 0_{m+1,2n-1}& -c^{-1} A^\intercal P^\intercal P A 
\end{pmatrix}=\\
\begin{pmatrix} I_{2n-1} &0_{2n-1, m+1} 
\\- c^{-1} A^\intercal P^\intercal &I_{m+1} \end{pmatrix} 
\begin{pmatrix} cI_{2n-1} &PA\\ A^\intercal P^\intercal & 0_{m+1} 
\end{pmatrix} 
\end{multline}
where 
$PA$ is an $(2n-1)\times (m+1)$ matrix, 
$I_{2n-1}$ is the $(2n-1)\times (2n-1) $ identity matrix,
$I_{m+1}$ is the $(m+1)\times (m+1) $ identity  matrix,
$0_{m+1}$ is the $(m+1)\times (m+1) $ zero   matrix
and
$0_{2n-1,m+1}$ is the $(2n-1)\times (m+1) $ zero   matrix.

Thus, the rank of the curvature matrix
 is $2n-1+ \rank(PA)$ and the rank of  $PA$ the same as the rank of the $(2n-1)\times (m+1)$ matrix 
\begin{equation}
\label{curv submatrix}
\begin{pmatrix}
-t^{-1}P \ubar\alpha & P\ubar v^{(1)}&\dots&P\ubar v^{(m)} 
\end{pmatrix} \text{ with $\ubar v^{(k)} 
= J_k^{\intercal} \ubar{\alpha}-\La_k^\intercal\alpha_{d+1}.$}
\end{equation} 
We observe that the extended columns
$-t^{-1} \ubar \alpha$, $\ubar v^{(1)}$, ...., $\ubar  v^{(m)} $ are linearly independent vectors in $\bbR^{2n}$.
To see  this  let  $(w_{2n},\bar {w})\in\bbR^{1+m}$ be such that
$-t^{-1} \ubar \alpha w_{2n} + \sum_{k=1}^m \ubar v^{(k)}w_k=0$. 
This is equivalent with 
$ -t^{-1} \ubar \alpha w_{2n} +
\sum_{k=1}^m J_k^\intercal \ubar \alpha w_k= \sum_{k=1}^m \La_k^\intercal w_k$.
If 
$\bar  w=0$ then we must also have  $w_{2n}=0$ since $\ubar\alpha\neq 0$. We thus need to show that $\bar   w\neq 0$ leads to a contradiction.  Let $\bar \om= \bar w/\|\bar  w\|$.
Since $\alpha_{d+1} =\alpha_{2n}$ we get  
$\ubar \alpha=\alpha_{2n}(-t^{-1} \om_{2n} I -J^{\overline \om } )^{-1} (\La^{\overline \om})^\intercal$ and thus by Lemma \ref{lem:skew-sym} 
\[|\ubar \alpha| 
\le (\tfrac{\om_{2n}^2}{t^2} + \|(J^{\overline \om})^{-1}\|^{-2})^{-1/2}\| \La^{\overline  \om} \| |\alpha_{2n}|
\le \|(J^{\overline\om} )^{-1}\|
\| \La^{\overline  \om} \|\, |\ubar \alpha|
\]
Since by assumption $\|\La^{\bar  \om}\|< \|(J^{\bar \om}) ^{-1}\|^{-1}$ for $|\bar   \om|=1$ we get $\ubar  \alpha=0$, a contradiction.

We have thus verified that the $m+1$ vectors 
$-t^{-1} \ubar \alpha$, $\ubar  v^{(1)}$, ..., $\ubar v^{(m)} $ are linearly independent
and hence the rank of the matrix \eqref{curv submatrix} is at least $m$. 
This proves \eqref{eq:rank-curv} and finishes the proof of Proposition
\ref{prop:Q5spacetimeTk}.

\section{ Proof of Theorem  \ref{thm:sphmeans}}
\label{sec:proofoftheoremaverage}
Let $\sigma (x,t,y',\theta_{2n}, \bar \theta)=  \theta_{2n}+(\ubar x^\intercal J^{\bar \theta} - t\La^{\bar \theta} )e_{2n} $ 
which is comparable to the `rotational curvature' of the fixed time operator. 
We use the oscillatory representation of the kernel in \eqref{eq:CalAk} and split $\cA^{k}_{t} =\sum_{\ell=0}^{[k/3] -1 }\cA^{k,\ell}_{t} +\widetilde \cA^{k, [k/3]}_t $ where 
\Be\label{eq:CalAkell}  \mathcal{A}^{k,\ell}_{t} f(x) = 2^{k(m+1)} \int_{\bbR^d}\int_{\R^{m+1}} e^{i2^k \Psi(x,t,y',\theta_{2n},\bar \theta)} b_\ell(x,t,y,\theta) \, d\theta_{2n}\,d\bar \theta\,f(y)\,dy, \Ee
and $b_\ell$ is supported where $|\theta_{2n} |\approx  2^{ -\ell} $ when $\ell\le [k/3] -1$ and supported where $|\theta_{2n}|\lc 2^{-k/3}$ if $\ell=[k/3]$.

The operators $\cA^{k,\ell}_t$ are bounded on $L^1$ and $L^\infty$ uniformly in $k$ and $\ell$. 
A trivial kernel estimate yields 
\Be \label{eq:L1Linfty}\|\cA^{k,\ell}_t f \|_\infty  \lc 2^{k(m+1)} 2^{-\ell}  
\|f\|_1 
\Ee 
We also have the $L^2$ estimates 
\Be \label{eq:L2L2}\|\cA^{k,\ell}_t f \|_2 \lc 2^{-k\frac{d-m-1}{2}} 2^{\ell/2} \|f\|_2
\Ee 
for $\ell\le[k/3].$  These follow, after an application of Plancherel's theorem, from corresponding bounds for the oscillatory integral operators with phase function $\Phi$ as in \eqref{eq:Phi-phase} 
\[
 T_{k,\ell} f(x,t) = 2^{k(m+1)} \int_{\bbR^d} e^{i2^k \Phi(x,t,y)} b_\ell(x,t,y) f(y)\, \,dy,\]
 namely 
 \Be \label{eq:fold-oscint} \|T_{k,\ell} f(\cdot,t) \|_2 \lc 2^{\ell/2} 2^{-kd/2} \|f\|_2.\Ee
The estimate \eqref{eq:fold-oscint} follows from bounds in  \cite{Cuccagna1997} (cf. Remark \ref{kernel-cokernel-rem}). 
Interpolation of the trivial $L^1$ estimate and \eqref{eq:L2L2} and summing in $\ell$ yields an $L^p\to L^p$ estimate $\|\cA^k_t\|_{L^p\to L^p} =O(2^{-\epsilon(p) k})$ with $\eps(p)>0$ for $1<p<\infty$. 

We may interpolate between \eqref{eq:L1Linfty} and \eqref{eq:L2L2} and obtain
\Be \label{eq:interpol-allp}
\|\cA^{k,\ell}_t f \|_q \lc  2^{k(m+1-\frac{d+m+1}{q})} 2^{\ell(\frac 3q-1)} \|f\|_{q'}  ,\quad 2\le q\le\infty
\Ee
which implies 
\Be 
\|\cA^{k}_tf\|_q \le C_q \begin{cases}
 2^{k(m+1-\frac{d+m+1}{q})} \|f\|_{q'}  ,&\quad 3< q\le\infty
 \\
 k 2^{k(m+1-\frac{d+m+1}{q})} \|f\|_{q'}  ,&\quad q=3
 \\
 2^{k(m+\frac 23-\frac{d+m}{q})} \|f\|_{q'}  ,&\quad 2\le q<3
 \end{cases}
\Ee 
For the case $q=3$, the Bourgain interpolation trick (as discussed in \S\ref{sec:interpolpf}) also yields
\Be \label{eq:q=3}
\|\cA^k_t f\|_{L^{3,\infty}} \lc 2^{k(m+1-\frac{d+m+1}{3})} \|f\|_{L^{\frac 32,1}}.
\Ee

In the case $m<2n-2$, we have $\frac{d+m+1}{m+1}>3$ and thus get a uniform estimate for $\cA^k_t$ when $q=\frac{d+m+1}{m+1}= \frac{2n+2m+1}{m+1}$. 
For $m=2n-2$, we have $\frac{d+m+1}{m+1}=3$ and obtain the restricted weak type $(q',q)$ estimate for $\cA^k_t$ uniformly in $k$. For  $m=2n-1$, we get  a uniform $L^{q'}\to L^q$-bound 
when $q=\frac{3(d+m)}{3m+2}=\frac{9m+3}{3m+2}$. 

To combine the $\cA^k_t$, we use standard applications of Littlewood-Paley theory, writing
$\cA^k_t = L_k\cA^k_t L_k +E_k$ where the $L_k $ satisfy Littlewood-Paley inequalities \[\Big\| \Big(\sum_{k\ge 0}|L_k f|^2\Big )^{1/2}\Big \|_r\lc \|f\|_r, \quad \Big\| \sum_{k\ge 0} L_k f_k\Big \|_r\lc \Big\|\Big(\sum_{k\ge 0} |L_k f_k|^2\Big)^{1/2}\Big  \|_r\]  for $1<r<\infty$ and the error term $E_k$ has $L^p\to L^q$ operator norm $O(2^{-k})$ for all $1\le p,q\le \infty$. Since $q'\le 2\le q$, a  standard application  of Littlewood-Paley inequalities in conjunction with Minkowski's inequalities allows us to deduce the endpoint estimate for $P_3$ when $m<2n-2$ and the $L^{q'}\to L^q$ bound for  $q=\frac{9m+3}{3m+2}$ for the case $m=2n-1.$ The inequalities for $(1/p,1/q)$ on the interior parts of the  edges $P_1P_3$ and $P_2P_3$  follow by interpolation.

When $q=3$ and $m=2n-2$, we still get uniform bounds for $\cA_t^k$ 
 on the interiors of  $P_1P_3$ and $P_2P_3$, 
and interpolating \eqref{eq:q=3} with $L^1\to L^1$ and $L^\infty\to L^\infty$ bounds gives us sharp $L^p\to L^q$ estimates for $\cA_t^k$  on these edges.  The above Littlewood-Paley trick still works for those $(p^{-1},q^{-1})$ on the open  edges which satisfy $p\le 2\le q$ and thus for those   $(p^{-1},q^{-1})$ we get the $L^p\to L^q$ boundedness for the averages. A further interpolation finishes the argument.

\begin{rem}
\label{rem:antidiag-max}
For the case $m+3<2n$ the  above estimates \eqref{eq:interpol-allp} also give a sharp result for the $L^{q'} \to L^q$ estimate for the full maximal operator in Theorem \ref{thm:main}, without imposing the condition \eqref{eq:smallness-Lambda} on $\Lambda$. By applying  \eqref{eq:Sobolev1D}  for $q$ in place of $p$ we get \Be \label{eq:interpol-allp-max}
\|\sup_{t\in [1,2]} |\cA^{k,\ell}_t f|  \|_q \lc  2^{k(m+1-\frac{d+m}{q})} 2^{\ell(\frac 3q-1)} \|f\|_{q'}  ,\quad 2\le q\le\infty,
\Ee which implies $\|\sup_{t\in [1,2]} |\cA^{k}_t f|  \|_q \lc  2^{k(m+1-\frac{d+m}{q})} \|f\|_{q'} $ for $q>3$ and hence the $L^{q_0',1}\to L^{q_0,\infty}$ bound for the maximal operator $M$ for $q_0= \frac{d+m}{m+1}=\frac{2n+m}{m+1}$, provided that $\frac{d+m}{m+1}>3$, i.e. $m+3<2n$. Moreover one obtains the $L^{q'}\to L^q$ bound for $M$ in the range  $2\le q<\frac{2(n+m)}{m+1}$ if  $m+3\le 2n$.
\end{rem}

Finally we consider $L^p\to L^q$ bounds for $p\ne q'$ in 
the case $m=2n-1$. To this end, we will now give a further estimate based on $L^2\to L^q$ estimates for oscillatory integral operators in \cite{GreenleafSeeger1994}.

\begin{prop}\label{prop:L2-Lq-av} For $1\le p\le 2$,  { $\tfrac 1q=\tfrac{d-1}{d}(1-\tfrac 1p), $}
\Be\label{eqL2Lq-av} 
\|\cA^k_t f\|_q \lc 2^{-k( d-1-\frac{d+m}{p})}\|f\|_p. \quad 
\Ee
\end{prop} 
\begin{proof}
This follows by an interpolation between the trivial $L^1\to L^\infty$ estimate with operator norm $O(2^{k(m+1)} )$ and the $L^2\to L^{q_0}$ estimate with $q_0=\frac{2d}{d-1}$ and  operator norm $\lc 2^{-k (d/q_0-(m+1)/2)} = 2^{-k(d-m-2)/2}$. 
The $L^2\to L^{q_0} $ bound follows via  Plancherel's theorem  from the estimate
\Be\label{osc eqL2Lq-av} 
\|T_k f(\cdot,t)\|_{q_0}  \lc 2^{-k d/{q_0}} \|f\|_2, \quad q_0=\tfrac{2d}{d-1}. \Ee
This in turn  is a consequence of  \cite[Thm. 2.2]{GreenleafSeeger1994} once we show that
the  $d-1$ dimensional conic variety $\Sigma^{\mathrm{fold}}_{x,t}= \{\nabla_x\Phi(x,t,y): \sigma(x,t,y)=0\} $ is a $d-1$ dimensional cone with $d-2$ nonvanishing principal curvatures everywhere (with $d=2n+m$).

Let $\Xi$ be as in \eqref{eq:Xidef} and let $\varPi\Xi\in \bbR^d $ be the spatial component of $\Xi$ (omitting the last component from $\Xi$). 
Let \[y_{2n}=\fy_{2n}(\bar y):= (t\La^{\bar y}-x^\intercal J^{\bar y})e_{2n}\] denote the solution of the equation $\sigma(x,t,y)=0$. 
We define 
\begin{align*}\xi(x,t,y',\bar y)&=  \varPi\Xi(x,t,y',\fy_{2n}(\bar y), \bar y))\\&=
\begin{pmatrix}
PJ^{\bar y} P^\intercal y'- t g(\ttf)PJ^{\bar y}e_{2n}
\\
(t\La^{\bar y}-x^\intercal J^{\bar y}) e_{2n}+
e_{2n}^\intercal J^{\bar y} (P^\intercal y'-tg(\ttf) e_{2n})
\\
\bar y
\end{pmatrix}
\end{align*}
From \eqref{eq:rank-condition} we see that 
$\xi_{y_1}, \dots, \xi_{y_{2n-1} }, \xi_{y_{2n+1}},\dots,  \xi_{y_{2n+m}} $ are linearly independent,  
\detail{ or 
\[ \rank \begin{pmatrix} xi_{y_1} &\dots& \xi_{y_{2n-1} } & \xi_{y_{2n+1}} &\dots & \xi_{y_{2n+m}} 
\endpmatrix = 2n-1+m.
\]
}which establishes $\Sigma^{\mathrm{fold}}_{x,t}$ as a manifold of dimension $2n-1+m$.

\detail{We compute 
\[ \xi_{y_j}
=\begin{pmatrix} 
PJ^{\bar y}  e_j+\partial_j g(\ttf) PJ^{\bar y} e_{2n}
\\
e_{2n}^\intercal J^{\bar y}  e_j
\\
0
\end{pmatrix} ,\]
\[
\xi_{\bar y_i}=\begin{pmatrix} 
PJ_iP^\intercal y'-tg(\ttf)PJ_ie_{2n}
\\
(t\La_i-x^\intercal J_i)e_{2n}
+ e_{2n}^\intercal J_i (P^\intercal y'-tg(\ttf) e_{2n}
\\
e_i^m
\end{pmatrix}\,.
\]}
We compute for $j,k\in\{1,\dots,2n\}$ and $i,l\in \{1,\dots,m\}$, \[
\quad \xi_{y_jy_k} = \begin{pmatrix} -t^{-1} \partial_{jk} g(\ttf) P J^{\bar y}e_{2n}\\0\\ 0\end{pmatrix}, \quad \xi_{y_j\bar y_i}=\begin{pmatrix}
PJ_i e_j+ \partial_jg(\ttf)  PJ_ie_{2n}
\\
e_{2n}^\intercal J_i ( e_j +\partial_jg (\ttf) e_{2n})\\0\end{pmatrix}
\]
and $\xi_{\bar y_i\bar y_l}=0$.

Define the normal vector $\nu$ for $x'=y'$ by  $\nu^\intercal =(\ubar \alpha^\intercal ,\bar \alpha^\intercal) $, and let $\alpha'=P\ubar \alpha$; so that
$\nu^\intercal \xi_{y_j}|_{x'=y'}=0$ for $j=1,\dots 2n-1$ and hence 
$\ubar \alpha^\intercal  J^{\bar y}  e_j=0$. Since $J^{\bar y}$ is invertible this implies that either $\ubar \alpha^\intercal  J^{\bar y}  e_{2n}\neq 0$  or $\ubar \alpha=0$. But the latter possibility would also  imply $\bar\alpha=0$ from the conditions $\inn{\nu}{\xi_{\bar y_i}}=0$.  Hence we have $\gamma:= \ubar \alpha^\intercal  J^{\bar y} e_{2n}\neq 0$. 

Let $\cC$ denote the $(2n-1+m)\times (2n-1+m)$ curvature matrix with respect to the normal $\nu$, with entries
$\inn{\nu} {\xi_{y_jy_k} }$ where $j,k\in \{1,\dots, 2n+m\}\setminus\{2n\} $.
When $x'=y'$ it  is given by
\begin{align*}\cC|_{x'=y'}=\inn\nu{\xi_{yy}''} 
 &=\begin{pmatrix} -t^{-1} \ga   I_{2n-1} &PM\\ M^\intercal P^\intercal &0\end{pmatrix} \quad\text{ with } \gamma = \ubar \alpha^\intercal  J^{\bar y} e_{2n} \neq 0,
\end{align*}
where 
$M$ is the $2n\times m$ matrix with $m$ columns $\sum_{j=1}^{2n} (\alpha^\intercal J_ie_j)e_j=-J_i\ubar\alpha $  and hence  
$PM$ is the $(2n-1)\times m$ matrix with columns $-PJ_i \ubar \alpha$, $i=1,\dots,m$.
Using a  lower dimensional version of the factorization \eqref{eq:factorization}
we see that the rank of 
$\cC$ at $x'=y'$ is equal to  the rank of 
\[
\begin{pmatrix} \ga I_{2n-1} &PM\\ 0_{m,2n-1}& -\gamma ^{-1} M^\intercal P^\intercal P M  
\end{pmatrix}, \quad \gamma = \ubar \alpha^\intercal J^{\bar y} e_{2n}\]
that is, $\rank\,\cC|_{x'=y'}=2n-1+ m-1=d-2$. 
\end{proof}

\begin{proof}[Conclusion of the proof of Theorem \ref{thm:sphmeans}] It remains to finish the argument for the 
`off-diagonal' estimates in part (iii) of this theorem.  Note that an $L^p\to L^q$ estimate implies an $L^{q'}\to L^{p'}$ estimate since the dual operator is similar  with $J$ replaced by $-J$.  

Let 
$p_1=\frac{d+m}{d-1}$, and $q_1=\frac{d(d+m)}{(d-1)(m+1)}$. 
Since $d=2n+m\ge m+2$, we have $p_1\le 2$ and $q_1\ge \frac{2d}{d-1}.$ Proposition \ref{prop:L2-Lq-av} yields the  $L^{p_1}\to L^{q_1}$ boundedness of the operators $\cA^k_t$ with norm uniform in $k$. The Littlewood-Paley arguments above also allow us to deduce the $L^{p_1}\to L^{q_1}$ boundedness of $\cA_t$, since $p_1\le 2\le q_1$. For  $m=2n-1$, we have $d=2n+m=2m+1$, and in this case,
$(\tfrac{1}{p_1},\tfrac 1{q_1} )= (\tfrac {2m}{3m+1}, \tfrac{2m^2+2m}{6m^2+5m+1})$
and 
$(1- \tfrac{1}{q_1}, 1-\tfrac 1{p_1})=  (\tfrac{4m^2+3m+1}{6m^2+5m+1}, \tfrac{m+1}{3m+1})$.
\end{proof} 

\section{Necessary Conditions for maximal operators} \label{sec:sharpness}

We provide five counter-examples, corresponding to each edge of the quadrilateral $\mathcal{R}$ for the Heisenberg group $\H^n$ (in particular, $m=1$), and one for the point $Q_2$. These show the necessity of all  the conditions 
in Theorem \ref{thm:maximal} and of some of the conditions in Theorem \ref{thm:main}.
The first four are suitable modifications of those in \cite{SchlagSogge1997} for the Euclidean case, which were in turn adapted from standard examples for spherical means and maximal functions.  These examples will be presented for all M\'etivier groups.  {The fifth example seems to be   new; it replaces the Knapp type example in the Euclidean case.}

\subsection{The line connecting \texorpdfstring{$Q_1$ and $Q_2$}{Q1 and Q2}} This is the necessary condition $p\leq q$ imposed by translation invariance and noncompactness of the group $G$ (see \cite{hormander1960}  for the analogous argument in the Euclidean case).

\subsection{\texorpdfstring{The line connecting $Q_2$ and $Q_3$}{The line connecting Q2 and Q3}} \label{sec:Q2Q3}
Let $B_{\delta}$ be the ball of radius $\delta$ centered at the origin. Let $f_{\delta}$ be the characteristic function of $B_{10\delta}$. Then \[\|f_{\delta}\|_p\approx \delta^{(2n+m)/p}.\] Let $C_\circ:=10(1+\|\La\|+\max_i\|J_i\|)$. For $1\le t\le 2$ we consider the sets
\[ R_{\delta,t}:= \{(\ubar{x}, \bar x) : ||\ubar x|-t|\le \delta/C_\circ,\,|\bar x-t\La\ubar x|\le \delta/C_\circ\}.\]
Then $|R_{\delta,t} |\gc \delta^{m+1} $.
Let $\Sigma_{x,t}= \{\om \in S^{2n-1}: |\ubar x-t \om|\le \delta/4\} $ which has spherical measure $\approx \delta^{2n-1}$.

If $x\in R_{\delta,t}$ and $\om \in \Sigma_{x,t}$ then  $|\ubar x-t\om |\le \delta$ and 
\[|\bar x- t \ubar x^\intercal J  \om-t^2\La\om| \le 
|\bar x-t \La\ubar x |+|\ubar  x^\intercal J(t \om - \ubar x)| + t|\La(t \om- \ubar x)|  \le 3\delta\]
(here we have used the skew symmetry of the $J_i$). We get 
\[f_{\delta}*\sigma_t(\ubar{x},\bar{x})=\int_{S^{2n-1}} f_{\delta}(\ubar{x}-t{\om},\bar{x}-t\ubar{x}^\intercal\! J{\om}-t^2\La \om)\,d\sigma({\om})\gtrsim \delta^{2n-1}\]
 for  $x\in R_{\delta,t}$. 
Passing to the maximal operator we consider $|\ubar x|\in [1,2]$ and put $t(x)=|\ubar x|$. Then setting
\[R_\delta= \big\{x: 1\le |\ubar x|\le 2,\, \big|\bar x- |\ubar x|\Lambda \ubar x\big | \le \delta/C_\circ\big\} \] 
we have $| R_\delta|\gc\delta^{m} $ and  $|f_\delta*\sigma_{t(x)}(x)|\ge \delta^{2n-1} $ for $x\in R_{\delta}$.

This yields the inequality
\[\delta^{2n-1}\delta^{m/q}\lesssim \delta^{(2n+m)/p},\]
and consequently, the necessary condition
\begin{equation}
\label{ball ineq}
\frac{m}{q}+2n-1\geq \frac{2n+m}{p},  
\end{equation}
that is, $(1/p,1/q)$ lies on or above the line connecting $Q_2$ and $Q_3$.

\subsection{\texorpdfstring{The point $Q_2$}{The point Q2}} 
For $p=p_2:=\frac{2n}{2n-1}=\frac{d-m}{d-m-1}$ the $L^p\to L^p$ bound fails. Here one uses a modification of Stein's example \cite{SteinPNAS1976} for the Euclidean spherical maximal function.
One considers the function $f_\alpha$ defined by $f_\alpha(\ubar v,v_{2n+1})= |\ubar v|^{-\frac{2n}{p_2}} |\log |v||^{-\alpha}  $ for $|\ubar v|\le 1/2$, $|v_{2n+1} |\le 1$ which belongs to $L^{p_2}$ for $\alpha>1/p_2$. One finds that if $t(\ubar x)=|\ubar x|$ then for 
$\alpha<1$ the integrals $f*\sigma_{t(x)}(x)$ are $\infty$ on a set of positive measure. If one  choose $\alpha$ close to $1$ this also shows that $M$ does not map any of the Lorentz spaces $L^{p_2,q}$ for $q<\infty$ to $L^{p_2,\infty}$.
\subsection{\texorpdfstring{The line connecting $Q_1$ and $Q_4$}{The line connecting Q1 and Q4}} \label{sec:Q1Q4} For this line we just use the counterexample for the individual averaging operators, bounding the maximal function from below by an averaging operator. 
Given $t\in [1,2]$, let $g_{\delta,t}$ be the characteristic function of the set
 $\{(\ubar{y},\bar{y}): | |\ubar{y}|-t|\le C_0\delta, |\bar{y}-t\La\ubar y|\le C_0\delta\}$ with $C_0=10 \sum_{i=1}^m \|J_i\|$. Thus  $\|g_{\delta,t}\|_p\lesssim \delta^{(m+1)/p}. $ 

Let  $x=(\ubar{x},\bar{x})$ be such that  $|\ubar x|\le  \delta$ and $|\bar x-t\La \ubar x| \le  \delta$. For any $\om\in S^{2n-1}$, we have that
$t|\ubar{x}^\intercal J{\om}|\lesssim 2\delta$. Thus 
\begin{gather*} \big||\ubar{x}-t\om|-t\big|\le 2  \delta\\
\big|\bar x-t \ubar x^\intercal J\om-t^2\La \om- t\La (\ubar x-t\om) \big| \le |\bar x-t\La\ubar x|+t|x^\intercal J\om| \le C_0 \delta \end{gather*}
implying that $|g_{\delta,t}*\sigma_t(x)|\gtrsim 1$. 
This yields the inequality $\delta^{(2n+m)/q}\le \delta^{(m+1)/p} $ which leads to the necessary condition
\begin{equation}
\label{scaling ineq}
\frac {1}{q} \ge \frac{m+1}{2n+m} \frac 1p,   
\end{equation} that is, $(1/p,1/q)$ lies on or above the line connecting $Q_1$ and $Q_4$.

\subsection{\texorpdfstring{The line connecting $Q_3$ and $Q_4$,  $m=1$}{The line connecting Q3 and Q4, m=1}} \label{sec:Q3Q4}
We now consider the case $m=1$; after a change of variables we may assume that the skew symmetric matrix $J$ satisfies the Heisenberg condition $J^2=-I$.  Pick a unit vector $u\in \bbR^{2n}$  so that $\La^\intercal \in \bbR u$, and set $v= Ju/\|Ju\|$ (thus $\inn{u}{v}=0$). Let $V=\mathrm{span}\{u,v\}$  and  let $V^\perp$ denote the orthogonal complement of $V$ in $\bbR^{2n}$.  Finally, let $\pi, \pi_\perp$  be the orthogonal projection to $V, V^\perp$ respectively. Note that $J$ maps $V$  into itself, since $J^2=-I$,   and since $J$ is skew-symmetric it also maps $V^\perp$ into itself. 

For  sufficiently large $C_1$ (say,  $C_1=10(2+\|\La\|))$  and small $\delta \ll C_1^{-1}$ let \[Q_\delta=\{(\ubar y, y_d): |\pi_{\perp} (\ubar y)|\le C_1\delta^{1/2}, \,|\pi (\ubar y)|\le C_1\delta,\, |y_d|\le C_1\delta\}.
\] Let $f_\delta= \bbone_{Q_\delta}$, so that $\|f_\delta\|_p^p\lc \delta^{(2n-2)/2+3}$, i.e. $\|f_\delta\|_p\lc\delta^{(n+2)/p}$.

For $1\le t\le 2$, let 
\begin{multline*} R^t_\delta=\big\{ (\ux,x_d): |\pi_\perp( \ux)|\le \delta^{1/2},\,
\big||\pi( \ubar x)|-t\big|\le \delta,\, |x_d-t\La \ubar x|\le \delta,  \\
1/4 <\inn{\ubar x}{u}, \inn {\ubar x}{v}<3/4\big\}
\end{multline*} and let 
$R_\delta=\cup_{9/8\le t\le 15/8} R^t_\delta$.  Then $|R_\delta| \approx \delta^{(2n-2)/2 +2} \delta^{-1}=\delta^n$.

For $x\in R_\delta$, 
we  derive  a lower bound for  $\sigma_{t(x)}* f_\delta(x)$, setting $t(x):= |\pi(\ubar x)|$.
Let \[S_{x}=\{\om\in S^{2n-1}: |\pi_\perp(\om)|\le \delta^{1/2}, \big|\inn{\om}{u}- \frac{\inn{\ubar x}{u}}{|\pi(\ubar x)|}\big|\le \delta , \, \inn{ \om}{v}>0\} ,\]
which has spherical measure $\gc \delta^{(2n-2)/2 +1 }=\delta^n $.

For $x\in R_\delta$, $\om\in S_x$, $t(x)= |\pi(\ubar x)|$ we have 
\Be\label{eq:piperpest} |\pi_\perp(\ubar x-t(x)\om)|\le 3\delta^{1/2}\Ee and \[ |\inn{\ubar x-t(x)\om}{u} | \le \delta. \] 
Since $\om \in S^{2n-1}$,  $\inn{\om}{v}>0$ and 
$\inn{\ubar x}{u}, \inn{\ubar x}{v} \in [1/4, 3/4] $, 
we also get
\begin{multline*}
    \big|\inn{\ubar x-t(x)\om}{v} \big |\le 2\Big |\frac{\inn{\ubar x}{v}}{|\pi(\ubar x)|}-\inn{\om}{v} \Big |\le 8\Big| \frac{\inn{\ubar x}{v} ^2}{|\pi(\ubar x)|^2} - \inn{\om}{v}^2\Big| \\= 8\Big|1-  \frac{\inn{\ubar x}{u} ^2}{|\pi(\ubar x)|^2} -1+ \inn{\om}{u}^2+|\pi_\perp(\om)|^2\Big|\le 8
    \big|\inn{\om}{u}- \frac{\inn{\ubar x}{u}}{|\pi(\ubar x)|}\big|+8\delta^2 \le 9\delta.
\end{multline*}
Hence
\Be\label{eq:piest} |\pi(x-t(x)\omega) | \le 10 \delta,  \quad \om \in S_x.\Ee
Now, since $J$ acts on $V$ and $V^\perp$,
\[ x^\intercal J\omega= (x-t(x) \omega)^\intercal J\omega 
= 
(\pi(x-t(x) \omega))^\intercal J\omega + (\pi_\perp (x-t(x)\omega))^\intercal J (\pi_ \perp \omega) 
\] and thus from \eqref{eq:piperpest} and \eqref{eq:piest}.
\[ |x^\intercal J\om| \le \delta+ 3\delta^{1/2} \delta^{1/2}  =4\delta,  \text{   $\omega\in S_x$. }
\]

From this we finally we obtain, writing $x_d-t^2\La\om= x_d-t\La(\ubar x)+t\La(\ubar x-t\om)$ and using that $\La^\intercal\in V$, 
\[ |x_d-t(x)\ux^\intercal J\om -t(x)^2\La \om| \le |x_d-t(x) \La\ubar x|+t(x)\|\La\|
|\pi(\ubar x-t(x)\om)| +4\delta \le C \delta.\]

These inequalities imply  
\begin{align*} M f_\delta(x) \ge f_\delta*\sigma_{t(x)} (x) &=\int_{S_x} f_\delta (\ubar x-t(x)\om, x_d-t(x) \ubar x^\intercal J \om -t(x)^2 \La\om) d\sigma(\om) \\& \ge |S_x|\gc \delta^n \text{ for $x\in R_\delta$}.
\end{align*} Hence we get
\[\|Mf_\delta\|_q/\|f_\delta\|_p \gc |R_\delta|^{1/q} \delta^n \delta^{-(n+2)/p} \gc \delta^{n/q+n-(n+2)/p} 
\] and letting $\delta\to 0$, we obtain  the necessary condition
\begin{equation}
\label{knapp ineq}
\frac{n}{q}+n\geq \frac{n+2}{p},
\end{equation}
that is, the necessary condition for $m=1$ is that   $(1/p,1/q)$ lies on or above the line connecting $Q_3$ and $Q_4$.

\section{Necessary conditions for averaging operators} \label{sec:sharpness-II}

We now prove the necessity of the conditions in Corollary \ref{cor:sphmeans}, and  of some of the conditions in Theorem \ref{thm:sphmeans}.

\subsection{\texorpdfstring{Necessary condition  for $n\ge 2$ and $m\le 2n-2$}{Necessary condition in higher dimensions}} For $n\ge 2$ the sharpness of Theorem \ref{thm:sphmeans} follows from the considerations in \S \ref{sec:sharpness}.
Concerning the line $Q_2Q_3$ we use the example in \S\ref{sec:Q2Q3} to get 
 \[\|f_\delta*\sigma_t\|_q \ge \delta^{2n-1+ (m+1)/q-(2n+m)/p} \|f_\delta\|_p\]
which gives the  necessary condition $ \frac{2n+m}{p}- \frac{m+1}q \le 2n-1.$
The calculation in \S\ref{sec:Q1Q4} only involves the averaging operator and yields 
the necessary condition $\frac 1q\ge \frac{m+1}{2n+m} \frac 1p$. 

\subsection{Sharpness for \texorpdfstring{$n=1$}{n=1}} 
\label{sec:sharpness-n=1} Here we can assume by a change of variables that $\La=0$ and that $x^\intercal Jy=x_2y_1-x_1y_2$.  We now  consider the circular means on $G=\bbH^1$ given by 
 \[Af(x)=\int f(x_1-\cos s, x_2-\sin s, x_3-x_2\cos s+x_1\sin s)\, ds.\]
 We need  to prove  the necessary condition \Be \label{necH1}
 6(1/p-1/q)\le 1,\Ee  i.e. $(1/p,1/q)$ cannot lie below the line connecting the points $(1/2,1/3)$ and $(2/3,1/2)$.
 This is 
 in analogy with the situation for integrals along the moment curve $(s,s^2,s^3)$ in the Euclidean situation of  $\bbR^3$; there the operator is tested   on indicator  functions of $(\delta, \delta^2,\delta^3)$-boxes. We show how to modify that example in  our situation.
 
 Let $f_\delta$ be the indicator function of the parallelepiped \[P_\delta=\{(y: |y_1|\le (2\delta)^2,\,\, |y_2|\le 2\delta,\,\, |y_3+y_2|\le (2\delta)^3 \}
 \]
  and
 \[ V_\delta=\{x: |x_1-1|\le \delta^2,\,\, |x_2|\le \delta, |x_3| \le \delta^3\}. \]
For $|s|\le \delta$, and $x\in V_\delta$  we have
\begin{align*}
&|x_1-\cos s| \le |x_1-1|+ \frac{s^2}{2}+ s^3\le 3\delta^2
\\
&|x_2-\sin s|\le |x_2|+|\sin s| \le 2\delta
\end{align*} and 
\begin{multline*}
|x_3-x_2\cos s+x_1\sin s)+ (x_2-\sin s)|\\ \le 
|x_3|+ |x_2||1-\cos s|+|\sin s||x_1-1|\le 3\delta^3
\end{multline*}
Thus if $y= 
 (x_1-\cos s, x_2-\sin s, x_3-x_2\cos s+x_1\sin s)$ for  $0\le s\le \delta$ and $x\in V_\delta$, then $y\in P_\delta$ and thus $Af_\delta(x)\ge \delta$. Hence
 \[\|Af_\delta\|_q\ge 
  |V_\delta|^{1/q} \delta= \delta^{1+6/q}\]
 and since
 $\|f_\delta\|_p=|P_\delta|^{1/p} \lc \delta^{6/p}$, we obtain the necessary condition \eqref{necH1}.
 
 \section{Implications for sparse bounds}\label{sec:sparse}
 As mentioned in the introduction one  principal goal  of  \cite{BagchiHaitRoncalThangavelu}  was to derive   for the {global} maximal operator $\fM$ inequalities of the form 
\Be\label{sparse-bound}  \int_{\bbH^n} \fM f(x) w(x) dx \le C\sup\big\{  \La_{\cS, p_1,p_2} (f,w) :\,{\cS\, \mathrm{sparse}} \big\},
\Ee
where the supremum is taken over {\it sparse families } of nonisotropic Heisenberg cubes 
(see \cite{BagchiHaitRoncalThangavelu} for precise definitions and constructions) and  the sparse forms $\La_{\fS,p_1,p_2} $  are 
given by 
\[\La_{\cS,p_1,p_2} (f,w) = \sum_{S\in \cS} |S| \Big(\frac{1}{|S|} \int|f|^{p_1}\Big)^{1/p_1} \Big(\frac{1}{|S|} \int_S |f|^{p_2} \Big)^{1/p_2}.\]
Relying entirely on arguments in \cite{BagchiHaitRoncalThangavelu} and using our  $L^p\to L^q$ bounds  we can show 
that  the sparse bound  \eqref{sparse-bound} 
holds  if  $(1/p_1, 1-1/p_2)$ lies in the interior of the quadrilateral $Q_1Q_2Q_3Q_4$ in 
\eqref{quadrilateral} (or on the open line segment $Q_1Q_2$), a result which is sharp up to the boundary. 

For the proof of sparse bounds for the global maximal operator the relevance of $L^p\to L^q$ results of localized maximal functions was recognized by Lacey \cite{laceyJdA19} in his work on the Euclidean spherical maximal function.  
Here we mention that the recent paper \cite{BeltranRoosSeeger} gives  very general results about this correspondence for the  Euclidean geometry;   Theorem 1.4 of that paper  is of particular relevance here  (see  also \cite{conde-alonso-etal2} for some results in spaces of homogeneous type). Moreover we refer to  \cite{BeltranRoosSeeger} for  general results about necessary conditions. 

For the proof of \eqref{sparse-bound} we  use the argument in \cite{BagchiHaitRoncalThangavelu}. One needs  to supplement the $L^p\to L^q$ bounds for the local maximal operator by a mild regularity result, namely
\Be\label{eq:eps-regularity result} \sup_{|h|\le 1}|h|^{-\eps} \big\|\sup_{t\in [1,2] } |(\tau_h f-f) *\ci J\mu_t| \big\|_q \lc \|f\|_p
\Ee for  some $\eps>0$; here $\tau_h$ 
\detail{$\tau_h f(y)=f(y \cdot h^{-1})$}
is the right translation operator, i.e. $\tau_h f(y)= f(\ubar y-\ubar h, \bar y-\bar h-\ubar y^\intercal J\ubar h)$.  One also needs to verify a dual condition which in our case is implied by \eqref{eq:eps-regularity result} and the symmetry of the sphere.
If $(1/p,1/q)$ belongs to  the interior of the boundedness region in Theorem \ref{thm:main} then our approach   yields   \eqref{eq:eps-regularity result} with an $\eps(p,q)>0$.  To 
prove this  one needs to show, by the localization argument in the beginning of \S\ref{sec:main-results} and the subsequent dyadic  decomposition,
that the operator $\cA^k$ in \eqref{eq:CalAk}  satisfies
\Be\label{eq:regularity=spt}\|\cA^k (\tau_h f-f) \|_{L^q(\bbR^d\times[1,2])} \lc 2^{-k(\frac 1q+a(p,q))} (2^{k}|h|)^{\eps} \|f\|_p
\Ee
for $|h|\ll 1$  and 
 functions $f$ supported near the origin,
 with $a(p,q)>0$ in the interior of the boundedness region. By taking means it suffices to   prove this for $\eps=0$ and $\eps=1$. The case for $\eps=0$ is immediate from the already proven results. For the case $\eps=1$ we use a change of  variables, followed by the  fundamental theorem of calculus, and a  change of variable again, with the fact that $(\tau_{s\underline h} \ubar y)^\intercal J_i \ubar h
 = \ubar y^\intercal J_i \ubar h$ 
 to write 
 \begin{multline*} 2^{-k(m+1)} \cA^k [\tau_h f-f](x,t)\\= 
 \int_0^1 \int f(\tau_{sh}y) \int e^{i 2^k\Psi(x,t,y,\theta) } (2^k \beta_1+\beta_2)\Big|_{(h, x,t,y,\theta)}  d\theta\, dy\, ds\end{multline*}
 with 
\begin{align*}
      \beta_1(h,x,t,y,\theta) &= ib(x,t,y',\theta) \big[ \ubar h^\intercal \nabla_{\underline{y}} \Psi + 
 \bar h^\intercal \nabla_{\bar y}\Psi + \uy^\intercal J^{\nabla_{\bar y}\Psi } \ubar h  \big]_{(x,t,y,\theta)},\\
 \beta_2(h,x,t,y,\theta) &= (h')^\intercal \nabla_{y'} b|_{(x,t,y',\theta)}.
\end{align*}

 Thus, taking into account the explicit form of the phase function \eqref{phasedefn}, 
 one can reduce the case for $\eps=1$ in \eqref{eq:regularity=spt} to estimates for operators of the form \eqref{eq:CalAk} already handled (note that here $\nabla_{\bar y} \Psi=-\theta$).

Finally, by similar arguments one gets  
 the regularity result for fixed $t$, 
\[ \|\cA_t^k (\tau_h f-f) \|_{L^q(\bbR^d)} \lc 2^{-kb(p,q)} (2^k|h|)^{\eps} \|f\|_p
\]
where $b(p,q)>0$ in the interior of the boundedness region in Corollary \ref{cor:sphmeans}. Again, using the reasoning in \cite{BagchiHaitRoncalThangavelu} this yields  
 an improved sparse bound for the lacunary maximal function, namely
\Be\label{sparse-bound-lac} \int_{\bbH^n} \sup_{k\in\Z}|f*\mu_{2^k} (x)|
w(x) dx \le C\sup
\big\{ \La_{\cS, p_1,p_2} (f,w):\, {\cS\, \mathrm{sparse}} \big\}\Ee
whenever $(1/p_1, 1-1/p_2)$ belongs to the interior of the boundedness region in Corollary \ref{cor:sphmeans}.

\begin{remarka}
The reader may wonder whether it is necessary to use the sparse bounds as in  \cite{BagchiHaitRoncalThangavelu} for the proof of  $L^p(\bbH^n)\to L^p(\bbH^n)$  bounds for the lacunary spherical maximal function, for  $1<p\le\infty$ and $n\ge 1$. We are grateful to both Luz Roncal and an anonymous referee for raising this question.  Indeed a more direct proof can be given; on can for example modify  the arguments in \cite{MuellerSeeger2004};  alternatively one can rely on a straightforward modification of the Calder\'on-Zygmund arguments in
\cite[\S6]{AndersonCladekPramanikSeeger}.
\end{remarka}

\bibliographystyle{amsplain}
\providecommand{\bysame}{\leavevmode\hbox to3em{\hrulefill}\thinspace}
\providecommand{\MR}{\relax\ifhmode\unskip\space\fi MR }
\providecommand{\MRhref}[2]{%
  \href{http://www.ams.org/mathscinet-getitem?mr=#1}{#2}
}
\providecommand{\href}[2]{#2}

\end{document}